\documentclass[11pt]{article}
\usepackage{geometry}
\usepackage{color}
\usepackage{float}
\usepackage{textcomp}
\usepackage{amsthm}
\usepackage{amsmath}
\usepackage{amssymb}%
\usepackage{mathrsfs}
\usepackage{multirow}
\usepackage{changes}
\usepackage{bm}

\usepackage{mathtools}
\usepackage{booktabs}    
\usepackage{subfigure} 
\usepackage{caption} 
\usepackage{multirow}
\usepackage{multicol}    
\usepackage{array}       
\usepackage{graphicx}
\graphicspath{{figuras/}}
\usepackage{empheq,amsmath}
\usepackage{hyperref}

\usepackage
[ 
lined, 
boxruled, 
commentsnumbered
]
{algorithm2e}
\DecMargin{0.1cm} 

\newtheorem{theorem}{Theorem}[section]
\newtheorem{lemma}[theorem]{Lemma}
\newtheorem{corollary}[theorem]{Corollary}

\newtheorem{remark}[theorem]{Remark}

\newtheorem{example}[theorem]{Example}

\DeclareMathOperator{\dom}{dom}

\newcommand{\R}{\mathbb{R}}
\newcommand{\E}{\mathbb{E}}

\newcommand{\la}{{\langle}}
\newcommand{\ra}{{\rangle}}

\newcommand{\bi}{\begin{itemize}}
\newcommand{\ei}{\end{itemize}}
\newcommand{\ba}{\begin{array}}
\newcommand{\ea}{\end{array}}

\def\beq{\begin{equation}}
\def\eeq{\end{equation}}
\def\ba{\begin{array}}
\def\ea{\end{array}}
\def\beann{\begin{eqnarray*}}
\def\eeann{\end{eqnarray*}}
\def\bea{\begin{eqnarray}}
\def\eea{\end{eqnarray}}

\def\Def{\stackrel{\mathrm{def}}{=}}

\newcommand{\mat}[1]{\bm{#1}}

\providecommand{\0}{\boldsymbol{0}}

\renewcommand{\aa}{\boldsymbol{a}}
\providecommand{\bb}{\boldsymbol{b}}

\providecommand{\ee}{\boldsymbol{e}}

\providecommand{\hh}{\boldsymbol{h}}

\renewcommand{\ss}{\boldsymbol{s}}

\providecommand{\uu}{\boldsymbol{u}}
\providecommand{\vv}{\boldsymbol{v}}

\providecommand{\xx}{\boldsymbol{x}}
\providecommand{\yy}{\boldsymbol{y}}
\providecommand{\zz}{\boldsymbol{z}}

\providecommand{\cO}{\mathcal{O}}

\providecommand{\elll}{\boldsymbol{\ell}}

\begin{document}

\title{\textbf{ Minimizing Quasi-Self-Concordant Functions\\by Gradient Regularization
		 of Newton Method}}

\author{
	Nikita Doikov \thanks{\'{E}cole Polytechnique F\'{e}d\'{e}rale de Lausanne (EPFL),
		Machine Learning and Optimization Laboratory (MLO), Switzerland (nikita.doikov@epfl.ch).\\[5pt]
The work was supported by the Swiss State Secretariat for Education, Research and
Innovation (SERI) under contract number 22.00133.  }  
}

\date{ August 28, 2023 }

\maketitle

\begin{abstract}
We study the composite convex optimization problems
with a Quasi-Self-Concordant smooth component.
This problem class naturally interpolates between 
classic Self-Concordant functions and functions with Lipschitz continuous Hessian.
Previously, the best complexity bounds for this problem class were associated with trust-region schemes
and implementations of a ball-minimization oracle.
In this paper, we show that for minimizing Quasi-Self-Concordant functions
we can use instead the basic Newton Method with Gradient Regularization. 
For unconstrained minimization, it only involves  
a simple matrix inversion operation (solving a linear system) at each step.
We prove a fast global linear rate for this algorithm, matching the
complexity bound of the trust-region scheme, 
while our method remains especially simple to implement.
Then, we introduce the Dual Newton Method, and based on it, develop the corresponding Accelerated 
Newton Scheme for this problem class, which further improves the complexity factor of the basic method.
As a direct consequence of our results, we establish fast
global linear rates of simple variants of the Newton Method
applied to several practical problems, including
Logistic Regression, Soft Maximum, and Matrix Scaling,
without requiring additional assumptions on strong or uniform convexity
for the target objective.
\end{abstract}


\section{Introduction}

\subsection{Motivation.}

The Newton Method is a fundamental algorithm in Continuous Optimization.
Involving in its computations the Hessian of the objective,
it is considered to be a very powerful approach for
solving numerical optimization problems.
The modern theory of second-order optimization methods
heavily relies on the notion of a \textit{problem class},
that formally postulates our assumptions on the target objective function,
and which determines the global complexity guarantees of the method.

\paragraph{Self-Concordant Functions.} Historically, one of the most successful descriptions of the Newton Method
is the class of \textit{Self-Concordant functions} \cite{nesterov1994interior},
which has become a universal tool for constructing efficient 
Interior-Point Methods for non-linear Convex Optimization.
A convex function $f$ is called Self-Concordant,
if its third derivative is bounded with respect to the \textit{local norm}
induced by the Hessian.
That is, for any $\xx \in \dom f$ and arbitrary direction $\uu$:
\beq \label{SelfConcordant}
\ba{rcl}
D^3f(\xx)[\uu, \uu, \uu] & \leq & M_{\mathfrak{sc}} \la \nabla^2 f(\xx) \uu, \uu \ra^{3/2}, 
\ea
\eeq
where $M_{\mathfrak{sc}} \geq 0$ is a certain constant (the parameter of self-concordance).
Note that the exponent $3/2$ on the right hand side of \eqref{SelfConcordant} is the only viable choice,
as dictated by the requirement for homogeneity in $\uu$.

It seems that Self-Concordant functions are likely
the most suitable problem class
for characterizing the behaviour of the classic Newton Method \cite{nesterov2018lectures}:
\beq \label{NewtonStep}
\ba{rcl}
\xx_{k + 1} & = & \xx_k - \gamma_k \nabla^2 f(\xx_k)^{-1} \nabla f(\xx_k),
\qquad k \geq 0.
\ea
\eeq
Here, $\gamma_k \in (0, 1]$ is a \textit{damping} parameter.
The value $\gamma_k = 1$ corresponds to the pure Newton step,
which minimizes the full second-order Taylor polynomial
of $f(\cdot)$ around the current point~$\xx_k$.
However, the pure Newton step may not work globally, even if the objective 
functions is strongly convex (see, e.g., Example 1.4.3 in \cite{doikov2021new}).

An important feature of both definition \eqref{SelfConcordant}
and the Newton step \eqref{NewtonStep}
is their invariance under affine transformation of variables:
they do not depend on artificial norms,
nor do they require a specific choice of the coordinate system.

For iterations of the damped Newton Method \eqref{NewtonStep}
minimizing a Self-Concordant function $f$,
we can prove both super-fast \textit{local quadratic convergence} 
(an affine-invariant variant of classical results about
the local convergence of the Newton Method \cite{fine1916newton,bennett1916newton,kantorovich1948newton}),
and \textit{global convergence},
provided an appropriate choice of the damping parameter $\gamma_k$.
Global convergence ensures 
that the method will reach the local region after a bounded number of iterations, starting from an arbitrary initial point $\xx_0$.
These global guarantees are essential for optimization algorithms, as the 
initial point may often be far from the solution's neighbourhood.
In the past decade, the study of the global behaviour of second-order methods
has become one of the driving forces in the field,
including analysis for Self-Concordant functions \cite{dvurechensky2018global,hanzely2022damped,nesterov2023set}.

While serving as a powerful tool for minimizing the barrier function in
Interior-Point Methods~\cite{nesterov1994interior}, it has become evident 
that the damping Newton scheme \eqref{NewtonStep}
and the corresponding problem class of Self-Concordant functions
do not fit all modern applications. For more refined problem classes,
it is possible to construct faster second-order schemes \cite{dvurechensky2018global,nesterov2018lectures}.
Thus, to fulfil all practical needs, different assumptions must be considered.

\paragraph{Lipschitz Continuous Hessian.}
In recent years, significant interest in second-order optimization
has been linked to the Cubic Regularization of Newton Method
with its excellent global complexity guarantees for a wide class of problems \cite{nesterov2006cubic}.
The core assumption of the Cubic Newton Method
is that the \textit{Hessian is Lipschitz continuous},
which can be expressed in terms of the third derivative,
for any $\xx \in \dom f$ and arbitrary direction $\uu$:
\beq \label{LipHessian}
\ba{rcl}
D^3 f(\xx)[\uu, \uu, \uu] & \leq & L_2 \| \uu \|^3.
\ea
\eeq
Here, $L_2 \geq 0$ represents the Lipschitz constant of the Hessian, and $\| \cdot \|$ denotes a fixed \textit{global norm}
(typically the standard Euclidean norm). In contrast to \eqref{SelfConcordant},
definition \eqref{LipHessian} is no longer affine-invariant, as it requires
choosing the norm in our primal space. 

Then, every step of the Cubic Newton
is minimization of a natural \textit{global upper model} of the target objective
around the current point $\xx_k$, which follows directly from the definition of our new problem class \eqref{LipHessian}.
This global upper model is just the full second-order Taylor polynomial augmented by the cubic regularizer.
Using the standard Euclidean norm in \eqref{LipHessian}, iterations
of the method can be written in the following canonical form:
\beq \label{RegNewton}
\ba{rcl}
\xx_{k + 1} & = & \xx_k - \bigl(  \nabla^2 f(\xx_k) + \beta_k \mat{I} \bigr)^{\!-1} \nabla f(\xx_k), \qquad k \geq 0,
\ea
\eeq
for some regularization parameter $\beta_k \geq 0$.
In the Cubic Newton Method, the value of $\beta_k$ is determined at each step
as a solution to a certain univariate non-linear equation, which ensures the global progress of every iteration.
For convex functions satisfying \eqref{LipHessian}, the global complexity of the Cubic Newton
is 
\beq \label{CNComplexity}
\ba{c}
\cO\Bigl( \bigl[ \frac{L_2 D^3}{\varepsilon} \bigr]^{1/2} \,\Bigr)
\ea
\eeq
second-order oracle calls to find an $\varepsilon$-solution
in terms of the functional residual, where $D$ is the diameter of the initial sublevel set.
Note that the complexity $\cO( \varepsilon^{-1/2} )$
is significantly better than the standard $\cO( \varepsilon^{-1} )$
of the first-order Gradient Method \cite{nesterov2018lectures}.
In subsequent works, various accelerated second-order schemes were discovered
\cite{nesterov2008accelerating,monteiro2013accelerated,nesterov2018lectures,doikov2020contracting,nesterov2021inexact,nesterov2021inexact2}
that achieve improved rates of convergence.
Eventually, in \cite{kovalev2022first,carmon2022optimal}, they attain the optimal complexity of $\cO( \varepsilon^{-2/7} )$,
which is the best possible \cite{arjevani2019oracle,nesterov2018lectures} for the problem class \eqref{LipHessian}.

Therefore, we can conclude that the second-order methods
for the convex functions with Lipschitz continuous Hessian~\eqref{LipHessian}
reached their natural theoretical limitations, and any further progress
in their convergence rates seems to require some fundamentally different assumptions.
At the same time, it is clear that the objective function
may belong to several problem classes simultaneously.
Thus, based on the cubic regularization, \textit{adaptive}
second-order methods that do not need to fix the Lipschitz constant
were studied in \cite{cartis2011adaptive1,cartis2011adaptive2},
and \textit{universal} schemes
that can automatically adjust to the problem classes with H\"older continuous Hessian
of arbitrary degree
were proposed in \cite{grapiglia2017regularized,grapiglia2019accelerated,doikov2021minimizing}.

An important recent line of research was devoted 
to the Gradient Regularization of Newton Method
\cite{polyak2009regularized,ueda2009regularized,mishchenko2023regularized,doikov2023gradient},
which proposes to choose the regularization parameter in method \eqref{RegNewton}
proportional to a certain power of the current gradient norm. That is,
for some $\sigma_k \geq 0$ and $\alpha \in [0, 1]$:
\beq \label{GradReg}
\ba{rcl}
\beta_k & := & \sigma_k \| \nabla f(\xx_k) \|^{\alpha}.
\ea
\eeq
It was first shown in \cite{mishchenko2023regularized} and independently rediscovered in \cite{doikov2023gradient},
that for the problem class \eqref{LipHessian}, we can set the parameters 
as follows:
$$
\ba{rcl}
\sigma_k & \equiv & L_2^{1/2} \qquad \text{and} \qquad \alpha \;\; \equiv \;\; 1/2,
\ea
$$
which equips the Gradient Regularization of Newton Method with
the same global complexity \eqref{CNComplexity} as for the Cubic Newton,
up to an additive logarithmic term.
Note that the rule \eqref{GradReg} is very simple to implement in practice.
After computing the regularization parameter $\beta_k$,  we need to perform just one matrix inversion operation
(solving a linear system), that is of the same cost as in the damped Newton scheme~\eqref{NewtonStep}.

\paragraph{Lipschitz Third Derivative.}

Later on, it was shown that
the Gradient Regularization technique makes the Newton Method
\textit{Super-Universal} \cite{doikov2022super} --- the algorithm is able to automatically adjust to a wide
family of problem classes
with H\"older continuous second and third derivatives.
The most notable example of this family is the class
of convex functions with \textit{Lipschitz continuous third derivative},
which was initially attributed to high-order Tensor Methods 
\cite{birgin2017worst,nesterov2019implementable,gasnikov2019near,grapiglia2020tensor,cartis2020sharp,doikov2021local}.
Taking into account convexity, this problem class can be characterized by the following bound
for the third derivative, for any $\xx \in \dom f$ and arbitrary directions $\uu, \vv$
(see Lemma~3 in \cite{nesterov2019implementable}):
\beq \label{LipThird}
\ba{rcl}
D^3 f(\xx)[\uu, \uu, \vv] & \leq &  \sqrt{2L_3} \la \nabla^2 f(\xx) \uu, \uu \ra^{1/2} \|\uu\| \|\vv\|,
\ea
\eeq
where $L_3 \geq 0$ is the Lipschitz constant.
It was shown in~\cite{doikov2022super} that we can set 
$$
\ba{rcl}
\sigma_k & \equiv & L_3^{1/3}
\qquad \text{and} \qquad
\alpha \;\; \equiv \;\; 2/3
\ea
$$
in the Gradient Regularization of Newton Method \eqref{RegNewton},\eqref{GradReg},
which gives the following global complexity:
\beq \label{ThirdDerGR}
\ba{rcl}
\cO\Bigl( \bigl[ \frac{L_3 D^4}{\varepsilon} \bigr]^{1/3} + \ln \frac{1}{\varepsilon} \,\Bigr)
\ea
\eeq
second-order oracle calls to find an $\varepsilon$-solution in terms of the functional residual.
We see that the dependence on $\varepsilon$ in this complexity bound is better 
than in \eqref{CNComplexity} of the standard Cubic Newton,
although the rate of convergence is still sublinear.
It is clear that such an advancement comes from using
 different assumption \eqref{LipThird}.

\paragraph{Quasi-Self-Concordance.}
In this work, we study the problem class 
of \textit{Quasi-Self-Concordant} convex functions,
that satisfy the following bound, for any
$\xx \in \dom f$ and arbitrary direction $\uu, \vv$:
\beq \label{QSC}
\ba{rcl}
D^3 f(\xx)[\uu, \uu, \vv] & \leq & 
M \la \nabla^2 f(\xx) \uu, \uu \ra  \| \vv \|,
\ea
\eeq
where $M \geq 0$
is the \textit{parameter of Quasi-Self-Concordance}.

This functional class was introduced in \cite{bach2010self},
extending the definition of the standard Self-Concordant functions \eqref{SelfConcordant}
to the Logistic Regression model.
Consequently it was studied in~\cite{sun2019generalized,karimireddy2018global,carmon2020acceleration},
in the context of generalized Self-Concordant functions and introducing the notion of \textit{Hessian stability}~\cite{karimireddy2018global}.
Inequality \eqref{QSC} can be seen as an interpolation
between the previous problem classes \eqref{SelfConcordant} and \eqref{LipHessian}.
It leads to the \textit{global} upper and lower second-order models of
the objective (see Figure~\ref{FigureModels} and Lemma~\ref{LemmaSmoothness}).

\begin{figure}[H]
	\centering
	\includegraphics[width=0.7\textwidth ]{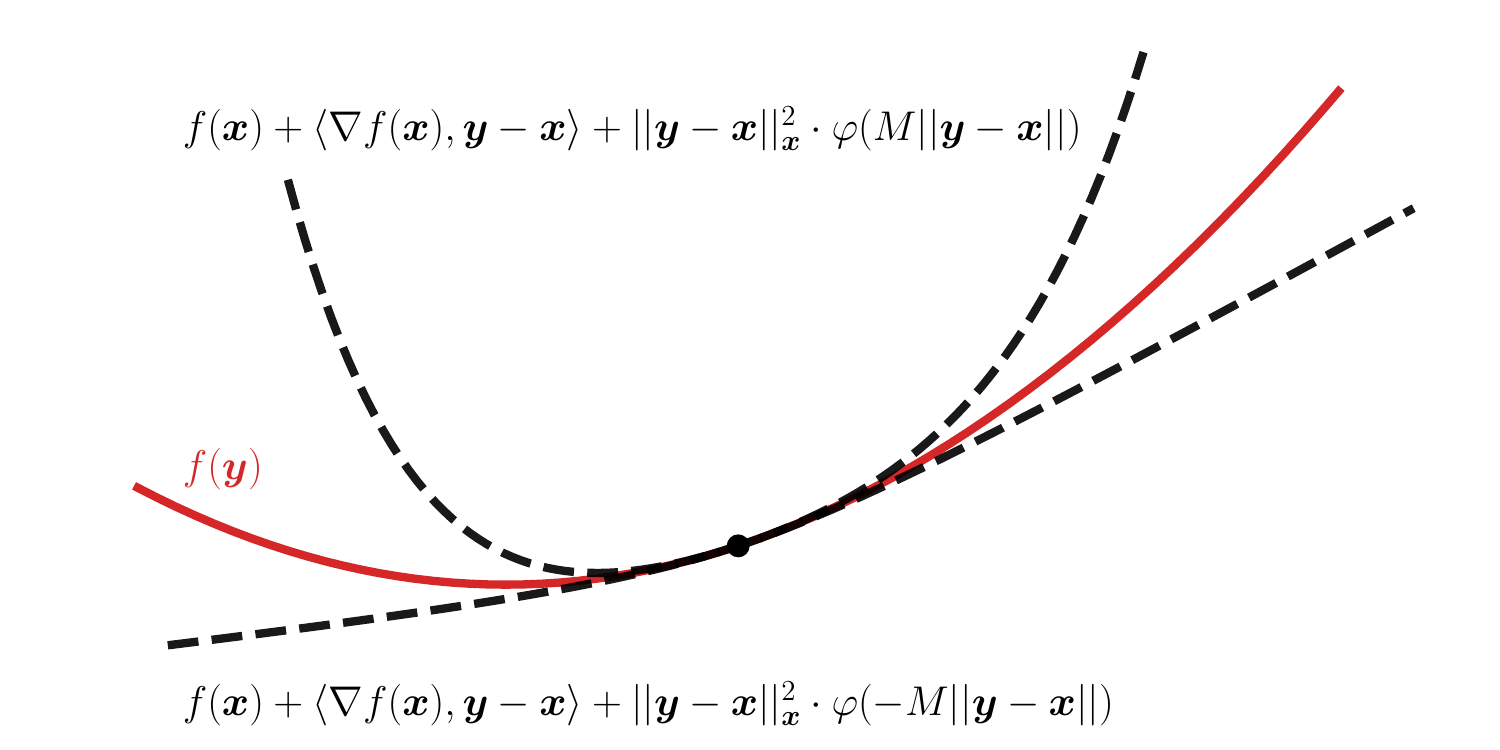}\\[2pt]
	\caption{ \small Global and lower second-order models
		of a Quasi-Self-Concordant function $f(\yy)$. 
	We denote $\| \yy - \xx \|_{\xx}^2 \Def \la \nabla^2 f(\xx)(\yy - \xx), \yy - \xx \ra$,
	and $\varphi(t) \Def \frac{e^t - t - 1}{t^2} \geq 0$ is a convex and monotone univariate function.}
	\label{FigureModels}
\end{figure}

It appears that many practical applications actually satisfy assumption \eqref{QSC},
including Logistic and Exponential Regression models \cite{bach2010self}, Soft Maximum \cite{nesterov2005smooth},
Matrix Scaling and Matrix Balancing \cite{cohen2017matrix} (see our discussion of these objective functions
and derivation of the corresponding parameter $M$ in Section~\ref{SectionQSC}).

At the same time, it seems to be a very suitable problem class for studying second-order optimization methods.
Previously, it was mainly considered for analysing the local behaviour 
of the damped Newton scheme \eqref{NewtonStep} in \cite{bach2010self,sun2019generalized}.
The global linear convergence on this problem class was established in \cite{karimireddy2018global}
for a variant of the trust-region method, and in \cite{carmon2020acceleration} 
for the methods based on implementations of a ball-minimization oracle.

\subsection{Contributions.}

In our paper, we show that a good choice for minimizing a Quasi-Self-Concordant function
is to employ the Gradient Regularization technique. Namely, we show that 
the Gradient Regularization of Newton Method \eqref{RegNewton}, \eqref{GradReg}
with the following selection of the parameters:
$$
\ba{rcl}
\sigma_k & \equiv & M \qquad \text{and} \qquad \alpha \;\; \equiv \;\; 1,
\ea
$$
achieves the \textit{global linear} rate of convergence, that matches the best rates of 
the trust-region schemes~\cite{karimireddy2018global,carmon2020acceleration}.
The simplest form of our method for unconstrained minimization,
choosing the standard Euclidean norm, is as follows (see algorithm~\eqref{PrimalNewton}):
$$
\boxed{
\ba{rcl}
\xx_{k + 1} & = & 
\xx_k  - \Bigl( \nabla^2 f(\xx_k) + M\| \nabla f(\xx_k) \| \mat{I} \Bigr)^{-1} \nabla f(\xx_k), \qquad k \geq 0.
\ea
}
$$
We prove that this method 
possesses the fast linear rate (Theorem~\ref{TheoremPrimalRate})
achieving the global complexity bound:
\beq \label{MainComplexity}
\ba{rcl}
\cO\Bigl( MD \cdot \ln \frac{1}{\varepsilon} \,\Bigr)
\ea
\eeq
iterations (second-order oracle calls) to find an $\varepsilon$-solution
in terms of the functional residual.
Note that for this result we do not require any extra assumption on strong or uniform 
convexity of the target objective.
Complexity bound \eqref{MainComplexity} is much better than both \eqref{CNComplexity} and \eqref{ThirdDerGR} 
in terms of the dependence on $\varepsilon$.
It shows that we can achieve any reasonable accuracy with a mild amount
of extra work of our method,
while the main complexity factor for solving the problem is the condition number $MD$.

In Section~\ref{SectionLocal}, we study the local behaviour 
of our method. We prove the local quadratic convergence
in terms of a new scale-invariant measure,
without assuming any uniform bound on the eigenvalues of the Hessian. Then,
we show the consequence of our theory for minimizing
strongly convex functions.

We use our local results in Section~\ref{SectionDual},
where
we present the Dual Newton Method.
This algorithm provides us with a \textit{hot-start}
possibility for minimizing the dual objects (the gradients).

Based on it, 
we develop the Accelerated Newton scheme for our problem class in Section~\ref{SectionAccelerated}.
It improves the complexity factor of the basic primal Newton method
by taking the power $2/3$ of the condition number. 
Totally, it needs the following number of second-order oracle calls
and the same number of quadratic minimization subproblems
to find an $\varepsilon$-solution:
\beq \label{AcceleratedComplexity}
\ba{rcl}
\tilde{\cO}\Bigl( (M R )^{2/3} \, \Bigr),
\ea
\eeq
where $R = \| \xx_0 - \xx^{\star}\| \leq D$
is the explicit distance from the initial point to the solution, and 
$\tilde{O}(\cdot)$ hides logarithmic factors.
This accelerated rate was also previously achieved using a different technique,
and the power $2/3$ was shown to be \textit{optimal} in the context of a ball-minimization oracle~\cite{carmon2022optimal}.

Therefore, our new accelerated algorithm
attains the best known dependence on the condition number \eqref{AcceleratedComplexity} for the problem class
of Quasi-Self-Concordant functions, while it remains especially easy to implement,
involving only a simple matrix inversion operation (solving a linear system) for unconstrained minimization,
and avoiding any difficult univariate equations.

\subsection{Notation.} \label{SubsectionNotation}

We denote by $\E$
a finite-dimensional vector space and by $\E^{*}$
its dual, which is the space of linear functions on $\E$.
The value of function $\ss \in \E^*$ at $\xx \in \E$ is denoted by
$\la \ss, \xx \ra \Def \ss(\xx)$.
Of course, one can always associate $\E$ and $\E^{*}$
with $\R^n$
by choosing a certain basis.  However, it is more natural to work directly with $\E$,
since all our results are independent of a coordinate representation of the objects.

Let us fix a self-adjoint positive-definite operator $\mat{B}: \E \to \E^{*}$
(notation $\mat{B} = \mat{B}^* \succ \0$)
and denote the following \textit{global} Euclidean norms by
$$
\ba{rcl}
\| \hh \| & \Def & \la \mat{B} \hh, \hh \ra^{1/2}, \qquad \hh \in \E, \\[10pt]
\| \ss \|_{*} & \Def & 
\max\limits_{\hh \in \E : \|\hh\| \leq 1} \la \ss, \hh \ra
\;\; = \;\; \la \ss, \mat{B}^{-1} \ss \ra^{1/2}, \qquad \ss \in \E^*.
\ea
$$
By associating $\E$ and $\E^{*}$ with $\R^n$, we can use the standard 
Euclidean norm with $\mat{B} := \mat{I}$ (identity matrix). However, for some problems we can have a better choice of $\mat{B}$, 
that takes into account geometry of the problem (see Examples~\ref{ExampleSoftMax}, \ref{ExampleGLM}).

For a several times differentiable function $f$,
we denote by $\nabla f(\xx)$ its gradient and 
by $\nabla^2 f(\xx)$ its Hessian, evaluated at some $\xx \in \dom f \subseteq \E$.
Note that
$$
\ba{rcl}
\nabla f(\xx), \; \nabla^2 f(\xx) \hh  & \in & \E^{*},
\qquad \hh \in \E.
\ea
$$

For convex functions, we have 
$\nabla^2 f(\xx) \succeq 0$. Let us denote by $\lambda(\xx) \geq 0$
the smallest eigenvalue of operator $\nabla^2 f(\xx)$ 
with respect to $\mat{B}$, thus
$$
\ba{rcl}
\lambda(\xx) & \Def & 
\max \Bigl\{ \, 
\lambda \geq 0 \; : \;  
\nabla ^2 f(\xx) - \lambda \mat{B} \succeq \0 
\,\Bigr\}
\;\; = \;\;
\lambda_{\min}\bigl( \mat{B}^{-1/2} \nabla^2 f(\xx) \mat{B}^{-1/2} \bigr),
\ea
$$
where the last expression can be seen
as the standard minimal eigenvalue of the corresponding matrix
for a certain choice of the basis
(while $\lambda(\xx)$ is invariant to the choice of coordinate system).

For a point   $\xx \in \dom f$ such that  $\nabla^2 f(\xx) \succ 0$,
we can also use the \textit{local primal norm} at point $\xx$:
$$
\ba{rcl}
\| \hh \|_{\xx} & \Def & \la \nabla^2 f(\xx) \hh, \hh \ra^{1/2}, 
\qquad \hh \in \E. 
\ea
$$
For convenience, we use notation $\| \hh \|_{\xx} $ even if 
the smallest eigenvalue is zero ($\lambda(\xx) = 0$).

We denote by $D^3 f(\xx)$
the third derivative of function $f$ at point $\xx \in \dom f \subseteq \E$,
which is a trilinear symmetric form. 
Thus, for any $\hh_1, \hh_2, \hh_3 \in \E$, we have $D^3 f(\xx)[\hh_1, \hh_2, \hh_3] \in \R$.
For the repeating arguments, we use the following shortenings:
$$
\ba{rcl}
D^3 f(\xx)[\uu]^2[\vv] & \equiv &
D^3 f(\xx)[\uu, \uu, \vv], \qquad 
D^3 f(\xx)[\uu]^3 \;\; \equiv \;\; D^3 f(\xx)[\uu, \uu, \uu],
\qquad \uu, \vv \in \E.
\ea
$$

Our goal is to solve the following Composite Convex Optimization problem:
\beq \label{MainProblem}
\ba{rcl}
F^{\star} \;\; = \;\;
\min\limits_{\xx \in Q} \Bigl\{  \;
F(\xx) & \Def & f(\xx) + \psi(\xx)
\; \Bigr\},
\ea
\eeq
where $Q \Def \dom \psi \subseteq \E$.
Function $f: Q \to \R$ is convex, several times differentiable and
\textit{Quasi-Self-Concordant}
(we formalize our smoothness assumption in the next section), and
$\psi: \E \to \R \cup \{ +\infty \}$
is a \textit{simple} closed convex proper function. 
We assume that a solution to \eqref{MainProblem} exists.

If $\psi \equiv 0$, then 
\eqref{MainProblem} is just unconstrained minimization problem: 
$
\min_{\xx} f(\xx).
$
Another classical example is $\psi$ being $\{0, +\infty \}$-indicator
of a given simple closed convex set $Q \subseteq \E$. Then, \eqref{MainProblem} is
constrained minimization problem:
$
\min_{\xx \in Q} f(\xx).
$

\section{Quasi-Self-Concordant Functions}
\label{SectionQSC}

We say that a convex function $f$ is Quasi-Self-Concordant with parameter $M \geq 0$,
if for all $\uu, \vv \in \E$ it holds
\beq \label{QSCDef}
\boxed{
\ba{rcl}
D^3 f(\xx)[\uu]^2[\vv] & \leq & M \| \uu \|_{\xx}^2 \| \vv \|, 
\qquad \xx \in \dom f.
\ea
}
\eeq

Let us start with several practical examples of objective functions,
that satisfy \eqref{QSCDef}. Then, we study the main properties
of this problem class.
Some of them are standard and can be found in the literature \cite{bach2010self,sun2019generalized}.
We provide their proofs for completeness of our presentation. 

\begin{example}[Quadratic Function] \label{ExampleQuadratic}
	Clearly,
	$$
	\ba{rcl}
	f(\xx) & := & \frac{1}{2} \la \mat{A} \xx, \xx \ra - \la \bb, \xx \ra,
	\ea
	$$
	for any $\mat{A} : \E \to \E^{*}$ such that $\mat{A} = \mat{A}^* \succeq 0$ 
	and $\bb \in \E^{*}$,
	satisfies \eqref{QSCDef} with $\boxed{M = 0}$.
\end{example}

\begin{example}[Soft Maximum] \label{ExampleSoftMax}
	For a given set of linear forms $\aa_1, \ldots, \aa_m \in \E^{*}$
	and vector $\bb \in \R^m$, let
	\beq \label{SoftMaxFunc}
	\ba{rcl}
	f_{\mu}(\xx) & := &	\mu \ln \biggl( \, \sum\limits_{i = 1}^m 
	e^\frac{\la \aa_i, \xx \ra - b^{(i)}}{\mu} \biggr)
	\quad \approx \quad \max\limits_{1 \leq i \leq m} \bigl[ \la \aa_i, \xx \ra - b^{(i)}  \bigr],
	\ea
	\eeq
	where $\mu > 0$ is a smoothing parameter.
	The problems of this type are important 
	in applications with minimax strategies for
	matrix games and $\ell_{\infty}$-regression~\cite{nesterov2005smooth}.
	Let us fix $\mat{B} := \sum_{i = 1}^m \aa_i \aa_i^{*}$. We assume $\mat{B} \succ \0$, otherwise
	we can reduce dimensionality of the problem.
	Then, denoting
	$$
	\ba{rcl}
	\pi_i & = & \pi_i(\xx) \;\; := \;\; 
	e^{ \frac{\la \aa_i, \xx \ra - b^{(i)}}{\mu} } 
	\biggl( \, \sum\limits_{i = 1}^m 
	 e^\frac{\la \aa_i, \xx \ra - b^{(i)}}{\mu} \biggr)^{-1},
	 \qquad 1 \leq i \leq m,
	\ea
	$$
	we have, by direct computation, for any $\uu, \vv \in \E$:
	$$
	\ba{rcl}
	\la \nabla f_{\mu}(\xx), \uu \ra & = & \sum\limits_{i = 1}^m \pi_i \la \aa_i, \uu \ra, \\
	\\
	\la \nabla^2 f_{\mu}(\xx) \uu, \vv \ra & = &
	\frac{1}{\mu}\sum\limits_{i = 1}^m \pi_i \la \aa_i, \uu \ra \la \aa_i, \vv \ra
	- \frac{1}{\mu} \la \nabla f_{\mu}(\xx), \uu \ra \la \nabla f_{\mu}(\xx), \vv \ra \\
	\\
	& = & \frac{1}{\mu} \sum\limits_{i = 1}^m \pi_i 
	\la \aa_i -  \nabla f_{\mu}(\xx), \uu \ra
	\la \aa_i - \nabla f_{\mu}(\xx), \vv \ra,
	\ea
	$$
	and
	$$
	\ba{rcl}
	D^3 f_{\mu}(\xx)[\uu]^2 [\vv]
	& = & 
	\frac{1}{\mu^2} \sum\limits_{i = 1}^m \pi_i
	\la \aa_i - \nabla f_{\mu}(\xx), \uu \ra^2 \la \aa_i - \nabla f_{\mu}(\xx), \vv \ra \\
	\\
	& \leq & 
	\frac{1}{\mu} \la \nabla^2 f_{\mu}(\xx)\uu, \uu \ra \max\limits_{1 \leq i, j \leq m} 
	\la \aa_i - \aa_j, \vv \ra
	\;\; \leq \;\; 
	\frac{2}{\mu} \la \nabla^2 f_{\mu}(\xx)\uu, \uu \ra \| \vv \|.
	\ea
	$$
	Hence, $f_{\mu}(\cdot)$ is Quasi-Self-Concordant
	with parameter $\boxed{ \textstyle M = \frac{2}{\mu}}$.
\end{example}

\begin{example}[Separable Optimization] \label{ExampleGLM}
	In applications related to Machine Learning and Statistics~\cite{hastie2009elements,sra2012optimization},
	very often we have the following structure of the objective function,
	for given data $\aa_1, \ldots, \aa_m \in \E^{*}$ and $\bb \in \R^m$:
	$$
	\ba{rcl}
	f(\xx) & = & \frac{1}{m} \sum\limits_{i = 1}^m \phi( \la \aa_i, \xx \ra - b^{(i)} ),
	\ea
	$$
	where $\phi: \R \to \R$ is a certain loss function.
	Some common examples are
	\begin{enumerate}
	\item Logistic Regression: $\phi(t) = \log(1 + e^t)$. We have
	$$
	\ba{rcl}
	\phi'(t) & = & \frac{1}{1 + e^{-t}}, \qquad
	\phi''(t) \;\; = \;\; \phi'(t) \cdot (1 - \phi'(t)), \\
	\\
	\phi'''(t) & = & \phi''(t) \cdot (1 - 2 \phi'(t)).
	\ea
	$$
	Therefore,
	$$
	\ba{rcl}
	| \phi'''(t) | & = & \phi''(t) \cdot | 1 - \frac{2}{1 + e^{-t}} |
	\;\; \leq \;\; \phi''(t), \qquad \forall t \in \R.
	\ea
	$$
	Thus, $\phi(\cdot)$ is Quasi-Self-Concordant with parameter $\boxed{M_{\phi} = 1}$.
	
	\item Exponential Regression and Geometric Optimization \cite{nesterov2018lectures}: $\phi(t) = e^t$. Clearly,
	$\phi'''(t)  =  \phi''(t) =  e^t$,
	and thus $\phi(\cdot)$ is Quasi-Self-Concordant with parameter $\boxed{M_{\phi} = 1}$.
	\end{enumerate}
Now, let us assume that the loss function $\phi(\cdot)$ is Quasi-Self-Concordant
with parameter $M_{\phi} \geq 0$,
and choose $\mat{B} := \sum_{i = 1}^m \aa_i \aa_i^*$. We have,
for any $\uu, \vv \in \E$:
$$
\ba{rcl}
\la \nabla^2 f(\xx) \uu, \uu \ra
& = & 
\frac{1}{m}\sum\limits_{i = 1}^m \phi''(\la \aa_i, \xx \ra - b^{(i)}) \la \aa_i, \uu \ra, \\
\\
D^3 f(\xx)[\uu]^2[\vv]
& = & \frac{1}{m}  \sum\limits_{i = 1}^m \phi'''(\la \aa_i, \xx \ra - b^{(i)}) \la \aa_i, \uu \ra^2 \la \aa_i, \vv \ra.
\ea
$$		
Hence,
$$
\ba{rcl}
D^3 f(\xx)[\uu]^2[\vv] & \leq & 
\frac{1}{m} \sum\limits_{i = 1}^m 
| \phi'''( \la \aa_i, \xx \ra - b^{(i)} ) | \la \aa_i, \uu \ra^2
\cdot \max\limits_{1 \leq i \leq m} | \la \aa_i, \vv \ra | \\
\\
& \leq & 
\frac{M_{\phi}}{m} \sum\limits_{i = 1}^m \phi''( \la \aa_i, \xx \ra - b^{(i)} )
\la \aa_i, \uu \ra^2 \cdot \| \vv \| \;\; = \;\; M_{\phi} \|\uu\|_{\xx}^2 \| \vv \|,
\ea
$$
and we conclude that $f$
is also Quasi-Self-Concordant with parameter $\boxed{M = M_{\phi}}$.
\end{example}

\begin{example}[Matrix Scaling and Matrix Balancing] \label{ExampleMatrixScalingBalancing}
	For a given non-negative matrix $\mat{A} \in \R^{n \times n}_+$,
	important for applications in Scientific Computing are the following objectives (see \cite{cohen2017matrix}),
	possibly with some additional linear forms.
	
	\begin{enumerate}
		\item Matrix Scaling:
		\beq \label{MatrixScaling}
		\ba{rcl}
		f(\xx, \yy) & = & \sum\limits_{1 \leq i, j \leq n} A^{(i, j)} e^{ x^{(i)} - y^{(j)} },
		\qquad \xx, \yy \in \R^n.
		\ea
		\eeq
		\item Matrix Balancing:
		\beq \label{MatrixBalancing}
		\ba{rcl}
		f(\xx) & = & \sum\limits_{1 \leq i, j \leq n} A^{(i, j)} e^{ x^{(i)} - x^{(j)} }, \qquad \xx \in \R^n.
		\ea
		\eeq
	\end{enumerate}
	Note that both of these objectives can be represented in the following general form.
	For a given set of positive numbers $a_1, \ldots, a_m \in \R_{+}$, and linear functions $\elll_1, \ldots, \elll_m \in \E^{*}$,
	we have:
	\beq \label{MatrixGeneral}
	\ba{rcl}
	f(\xx) & = & \sum\limits_{i = 1}^m a_i e^{ \la \elll_i, \xx \ra }, \qquad \xx \in \E,
	\ea
	\eeq
	which is Quasi-Self-Concordant due to the previous example.
	In both cases \eqref{MatrixScaling}, \eqref{MatrixBalancing}, we set $m := n^2$
	and $a_{i + n (j - 1)} := A^{(i, j)}$ for $1 \leq i, j \leq n$. Then,
	for Matrix Scaling \eqref{MatrixScaling} we choose
	\beq \label{ScalingInstance}
	\ba{rcl}
	\E & := &  \R^{2n} \qquad \text{and} \qquad 
	\elll_{i + n(j - 1)} \;\; := \;\; \ee_i - \ee_{n + j}, \qquad 1 \leq i, j \leq n,
	\ea
	\eeq
	where $\ee_1, \ldots, \ee_{2n}$ are the standard basis vectors in $\R^{2n}$. 
	For Matrix Balancing \eqref{MatrixBalancing}, we set
	\beq \label{BalancingInstance}
	\ba{rcl}
	\E & := & \R^n \qquad \text{and} \qquad
	\elll_{i + n(j - 1)} \;\; := \;\; \ee_i - \ee_j, \qquad 1 \leq i, j \leq n.
	\ea
	\eeq
	Note that for both instances \eqref{ScalingInstance} and \eqref{BalancingInstance},
	we can use $\mat{B} := \mat{I}$ (identity matrix of the appropriate dimension).
	Then, we get, for any $\uu, \vv \in \E$:
	$$
	\ba{rcl}
	D^3 f(\xx)[\uu]^2[\vv] & = & \sum\limits_{i = 1}^m a_i e^{\la \elll_i, \xx \ra}
	\la \elll_i, \uu \ra^2 \la \elll_i, \vv \ra 
	\;\; \leq \;\; 
	\sum\limits_{i = 1}^m a_i e^{\la \elll_i, \xx \ra}
	\la \elll_i, \uu \ra^2 \cdot \max\limits_{1 \leq i \leq m} | \la \elll_i, \vv \ra | \\
	\\
	& = & \| \uu \|_{\xx}^2 \cdot \max\limits_{1 \leq i \leq m} | \la \elll_i, \vv \ra | 
	\;\; \leq \;\; \sqrt{2} \| \uu \|_{\xx}^2  \| \vv \|.
	\ea
	$$
	Therefore, we conclude that both \eqref{MatrixScaling} and \eqref{MatrixBalancing}
	are Quasi-Self-Concordant with respect to the standard Euclidean norm with parameter $\boxed{M = \sqrt{2}}$.
\end{example}

Let us declare some of the most important properties of 
Quasi-Self-Concordant functions, that follow
directly from definition \eqref{QSCDef}.
We allow the following simple operations.

\begin{enumerate}
	\item \textit{Scale-invariance.} Note that multiplying Quasi-Self-Concordant function by an arbitrary positive coefficient $c > 0$:
	$$
	\ba{rcl}
	f(\cdot) &\mapsto& c  f(\cdot),
	\ea
	$$
	the parameter $M$ remains the same.
	
	\item \textit{Sum of two functions.}
	Let $f_1$ and $f_2$ be Quasi-Self-Concordant with some $M_1, M_2 \geq 0$. Then,
	the function 
	$$
	\ba{rcl}
	f(\cdot)  &= & f_1(\cdot) + f_2(\cdot)
	\ea
	$$ 
	is Quasi-Self-Concordant with parameter $M = \max\{ M_1, M_2 \}$.
	As a direct consequence, adding to our objective an arbitrary convex quadratic function
	does not change $M$.
	
	\item \textit{Affine substitution.}
	Let $\E_1$ and $\E_2$ be two vector spaces.
	Let $\mat{B}: \E_1 \to \E_1^*$ be a self-adjoint positive operator 
	that defines the global norm in $\E_1$.
	Assume that $f(\cdot)$, $\dom f \subseteq \E_1$ is Quasi-Self-Concordant with parameter $M_f$
	with respect to this norm. Consider the function:
	$$
	\ba{rcl}
	g(\xx) & = & f(\mat{A} \xx - \bb), \qquad \xx \in \E_2,
	\ea
	$$
	where $\mat{A}: \E_2 \to \E_1$ is a given linear transformation and $\bb \in \E_1$.
	Then, $g(\cdot)$ is Quasi-Self-Concordant 
	for the same parameter $M_g = M_f$
	with respect to the norm induced be the following operator:
	\beq \label{Bprime}
	\ba{rcl}
	\mat{B}' & = & \mat{A}^{*} \mat{B} \mat{A} \;\;  : \;\; \E_2 \to \E_2^*.
	\ea
	\eeq
\end{enumerate}

As opposed to the classic Self-Concordant functions \cite{nesterov1994interior},
we see that the Quasi-Self-Concordant functions \eqref{QSCDef} are \textit{scale-invariant}, but not \textit{affine-invariant}.
The latter problem class depends on a fixed global norm $\| \cdot \|$,
and after applying an affine substitution, we should modify the norm accordingly \eqref{Bprime}.

Let us study the global properties of the Quasi-Self-Concordant functions, that
we use intensively in the analysis of our methods. 
We start with the global behaviour of the Hessians.

\begin{lemma} \label{LemmaHessians}
For any $\xx, \yy \in \dom f$, it holds
\beq \label{HessBound}
\ba{rcl}
\nabla^2 f(\xx) e^{-M\|\yy - \xx\|}
\;\; \preceq \;\;
\nabla^2 f(\yy) & \preceq & \nabla^2 f(\xx) e^{M\|\yy - \xx\|}.
\ea
\eeq
\end{lemma}
\begin{proof}
	Let us fix a direction $\uu \in \E$ and consider a point $\xx \in \dom f$ s.t. 
	$\la \nabla^2 f(\xx) \uu, \uu \ra > 0$.
	By continuity, we know that
	for $\yy \not= \xx$ that is sufficiently close to $\xx$, we also have
	\beq \label{HessLocalPositive}
	\ba{rcl}
	\la \nabla^2 f(\xx + t (\yy - \xx)) \uu, \uu \ra & > &  0, \qquad \forall t \in [0, 1].
	\ea
	\eeq
	Then, the univariate function $g(t) = \ln \| \uu \|_{\xx + t(\yy - \xx)}^2$, $t \in [0, 1]$ is well defined, and we obtain
	$$
	\ba{rcl}
	|g'(t)| & = & \Bigl| \frac{ D^3 f(\xx + t(\yy - \xx))[\uu]^2[\yy - \xx] }{\|\uu\|_{\xx + t(\yy - \xx))}^2}   \Bigr|
	\;\; \overset{\eqref{QSCDef}}{\leq} \;\; M\|\yy - \xx\|.
	\ea
	$$
	Hence, from the Mean-Value Theorem, we conclude that
	$$
	\ba{rcl}
	\Bigl| \ln \frac{\| \uu \|_{\yy}^2}{ \| \uu \|_{\xx}^2} \Bigr|
	& = & 
	| g(1) - g(0) | \;\; \leq \;\; M\|\yy - \xx\|.
	\ea
	$$
	which is
	\beq \label{HessBoundU}
	\ba{rcl}
	\la \nabla^2 f(\xx) \uu, \uu \ra e^{-M\|\yy - \xx\|}
	& \leq & 
	\la \nabla^2 f(\yy) \uu, \uu \ra 
	\;\; \leq \;\;
		\la \nabla^2 f(\xx) \uu, \uu \ra e^{M\|\yy - \xx\|}.
	\ea
	\eeq
	Note that we established \eqref{HessBoundU} for all $\yy \in \dom f$ such that \eqref{HessLocalPositive} holds.
	
	Assume that for some $\yy' \in \dom f$, we have  $\la \nabla^2 f(\yy')\uu, \uu \ra = 0$.
	Then, we can consider 
	$$
	\ba{rcl}
	\alpha^* & = & \inf\bigl\{ \alpha > 0 \; : \; \la \nabla^2 f(\xx + \alpha(\yy' - \xx)) \uu, \uu \ra = 0  \bigr\}.
	\ea
	$$
	Thus, for $\yy^* := \xx + \alpha^{*}(\yy' - \xx)$, we also have $\la \nabla^2 f(\yy^{*}) \uu, \uu \ra = 0$.
	However, for any sufficiently small $\epsilon > 0$ there is a point $\yy^{*}_{\epsilon} :=  \xx + (\alpha^* - \epsilon)(\yy' - \xx) \in \dom f$ such 
	that the value of quadratic form is separated from zero:
	$$
	\ba{rcl}
	\la \nabla^2 f(\yy^*_{\epsilon}) \uu, \uu \ra  & \overset{\eqref{HessBoundU}}{\geq} &
	\la \nabla f(\xx)\uu, \uu \ra e^{-M\alpha^*} \;\; > \;\; 0,
	\ea
	$$
	which contradicts $\la \nabla^2 f(\yy^{*}) \uu, \uu \ra = 0$
	and hence contradicts $\la \nabla^2 f(\yy')\uu, \uu \ra = 0$.
	
	Therefore, we conclude that $\la \nabla^2 f(\yy)\uu, \uu \ra > 0$
	for \textit{all} $\yy \in \dom f$ which finishes the proof.
\end{proof}

\begin{remark} \label{HessPositiveDef}
	From the proof of Lemma~\ref{LemmaHessians}, we see that
	for any direction $\uu \in \E$, we can have either
	$$
	\ba{rcl}
	\la \nabla^2 f(\xx)\uu, \uu \ra  & > & 0, \qquad \forall \xx \in \dom f,
	\ea
	$$
	or
	$$
	\ba{rcl}
	\la \nabla^2 f(\xx)\uu, \uu \ra  & \equiv & 0, \qquad \forall \xx \in \dom f.
	\ea
	$$
	In the latter case, we can reduce dimensionality of the problem.
\end{remark}

Using a simple integration argument, we obtain 
now global bounds on the variation of the gradients,
and objective function values.

\begin{lemma} \label{LemmaSmoothness}
	For any $\xx, \yy \in \dom f$, it holds
	\beq \label{GradBound}
	\ba{rcl}
	\| \nabla f(\yy) - \nabla f(\xx) - \nabla^2 f(\xx)(\yy - \xx) \|_* & 
	\leq & M \|\yy - \xx\|_{\xx}^2 \cdot \varphi ( M \|\yy - \xx\| ),
	\ea
	\eeq
	and
	\beq \label{FuncBound}
	\ba{rcl}
	\|\yy - \xx\|_{\xx}^2 \cdot \varphi(-M\|\yy - \xx\|) 
	&  \leq &
	f(\yy) - f(\xx) - \la \nabla f(\xx), \yy - \xx \ra \\
	\\
	& \leq & \| \yy - \xx\|_{\xx}^2 \cdot \varphi(M\|\yy - \xx\|),
	\ea
	\eeq
	where $\varphi(t) \Def \frac{e^t - t - 1}{t^2} \geq 0$.
\end{lemma}
\begin{proof}
Indeed, by Taylor formula, we have for any $\vv \in \E$:
$$
\ba{cl}
& \la \nabla f(\yy) - \nabla f(\xx) - \nabla^2 f(\xx)(\yy - \xx), \vv \ra
\;\; = \;\;
\int\limits_0^1 (1 - \tau) D^3 f(\xx + \tau(\yy - \xx))[\yy - \xx]^2[\vv] d\tau \\
\\
& \overset{\eqref{QSCDef}}{\leq} \;
M \|\vv\|  \cdot
\int\limits_{0}^1 (1 - \tau) \|\yy - \xx\|^2_{\xx + \tau(\yy - \xx)} d \tau 
\;\; \overset{\eqref{HessBound}}{\leq} \;\;
M\|\yy - \xx\|_{\xx}^2  \cdot \| \vv \|  \cdot
\int\limits_{0}^1 (1 - \tau) e^{\tau M\|\yy - \xx\|} d\tau \\
\\
&  = \;
M \|\yy - \xx\|_{\xx}^2 \cdot  \| \vv \|  \cdot \varphi( M\| \yy - \xx\| ).
\ea
$$
Maximizing this bound with respect to $\vv : \| \vv \| \leq 1$ gives \eqref{GradBound}.
For the objective function, we get
$$
\ba{cl}
& f(\yy) - f(\xx) - \la \nabla f(\xx), \yy - \xx \ra
\;\; = \;\;
\int\limits_{0}^1 (1 - \tau) \nabla^2 f(\xx + \tau(\yy - \xx))[\yy - \xx]^2 d\tau \\
\\
& \overset{\eqref{HessBound}}{\leq} \;
\| \yy - \xx \|_{\xx}^2 \cdot \int\limits_{0}^1 (1 - \tau) e^{\tau M\|\yy - \xx\|} d\tau
\;\; = \;\;
\| \yy - \xx\|_{\xx}^2 \cdot \varphi(M\|\yy - \xx \|),
\ea
$$
which is the right hand side of \eqref{FuncBound}.
To prove the left hand side,  we use the lower bound for the Hessian from \eqref{HessBound} . It gives
$$
\ba{rcl}
 f(\yy) - f(\xx) - \la \nabla f(\xx), \yy - \xx \ra
&\overset{\eqref{HessBound}}{\geq} &
\| \yy - \xx \|_{\xx}^2 \cdot \int\limits_{0}^1 (1 - \tau) e^{-\tau M\|\yy - \xx\|} d\tau \\
\\
& = & \|\yy - \xx\|_{\xx}^2 \cdot \varphi(-M\|\yy - \xx\|),
\ea
$$
that completes the proof.
\end{proof}

\begin{remark}
	Function $\varphi(t) = \frac{e^t - t - 1}{t^2}$
	used in Lemma~\ref{LemmaSmoothness}
	is convex and monotone. Its graph is shown in Figure~\ref{FigurePhi}.
\end{remark}

\begin{figure}[H]
	\centering
	\includegraphics[width=0.32\textwidth ]{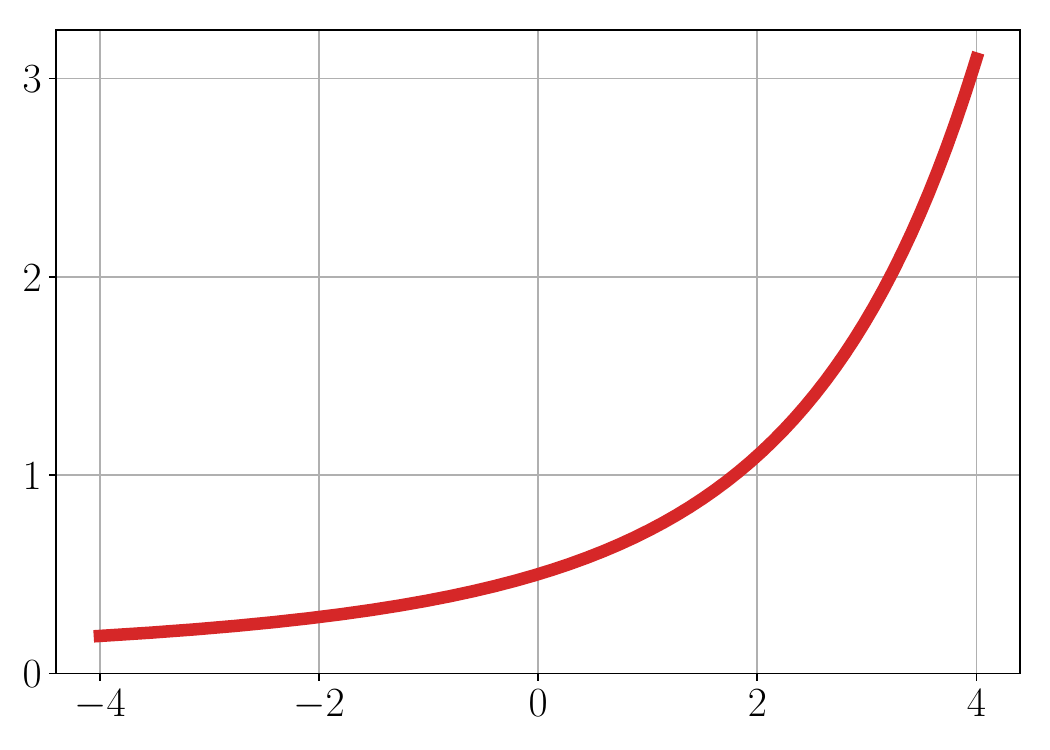}\\[2pt]
	\caption{ \small Graph of function $\varphi(t) = \frac{e^t - t - 1}{t^2}$.}
	\label{FigurePhi}
\end{figure}

	Note that \eqref{FuncBound} are the global upper and lower
	second-order models of the objective (see Figure~\ref{FigureModels}).
	It seems to be difficult to use them directly in optimization algorithms,
	since these models are non-convex in general.
	However, from Lemmas~\ref{LemmaHessians} and \ref{LemmaSmoothness} we see
	that function $f(\cdot)$ is almost quadratic in every ball 
	of radius $r := 1 / M$. Indeed, according to \eqref{HessBound}, for any $\xx, \yy \in \dom f$
	s.t. $\| \xx - \yy \| \leq 1/M$ the Hessian is \textit{stable} 
	\cite{karimireddy2018global}:
	$$
	\ba{rcl}
	\frac{1}{e}\nabla^2 f(\xx) & \preceq & \nabla^2 f(\yy)
	\;\; \preceq \;\; e \nabla^2 f(\xx).
	\ea
	$$
	Therefore, for minimizing a Quasi-Self-Concordant objective we 
	can use a trust-region scheme with a fixed radius 
	\cite{karimireddy2018global,carmon2020acceleration},
	that achieves the global complexity: 
	$\tilde{\cO}\bigl( MD \bigr)$
	iterations to solve the problem with any given accuracy (global linear rate of convergence),
	where $D$ is the diameter of the initial sublevel set and $\tilde{\cO}(\cdot)$
	hides logarithmic factors.
	
	As we show in the next section, we can use instead 
	the Gradient Regularization of Newton Method,
	which is easier to implement and, most importantly, it does not need to fix the value of $r = 1/M$
	(the trust-region radius), allowing an adaptive choice of parameter $M$.

\section{Gradient Regularization of Newton Method}
\label{SectionPrimal}

In this section, we propose the primal Newton scheme
for minimizing Quasi-Self-Concordant objectives
in the composite form \eqref{MainProblem}.
For a given $\xx \in Q$, we consider the point $\xx^+$ defined as
\beq \label{PrimalStep}
\boxed{
\ba{rcl}
\xx^+ & \Def &
\arg\min\limits_{\yy \in Q}
\Bigl\{ \;  \la \nabla f(\xx), \yy - \xx \ra + \frac{1}{2} \| \yy - \xx\|_{\xx}^2
	+ \frac{\beta}{2}\| \yy - \xx\|^2 + \psi(\yy)  \; \Bigr\}
\ea
}
\eeq
with regularization coefficient proportional to the current (sub)gradient norm: 
\beq \label{BetaChoice}
\boxed{
\ba{rcl}
\beta & := & \sigma \| F'(\xx) \|_*
\ea
}
\eeq
for some $F'(\xx) \in \partial F(\xx)$ and parameter $\sigma \geq 0$. 
Thus, \eqref{PrimalStep} is a simple Newton step
with \textit{quadratic} regularization and with a possible composite component.
In case $\psi \equiv 0$ (unconstrained minimization), we obtain a direct formula
for one step:
$$
\ba{rcl}
\xx^+ & = & \xx - \Bigl( \nabla^2 f(\xx) + \sigma \| \nabla f(\xx) \|_* \mat{B} \Bigr)^{-1} \nabla f(\xx),
\ea
$$
which requires solving just one linear system, as in the classical Newton Method.
In general, \eqref{PrimalStep} is a composite quadratic minimization problem,
and we can use first-order Gradient Methods~\cite{nesterov2018lectures} for solving it.
Note that the objective in \eqref{PrimalStep}
is \textit{strongly convex} for $\beta > 0$,  and we can expect fast linear rate of convergence for 
these gradient-based solvers.

Optimality condition for one step \eqref{PrimalStep} is as follows (see, e.g., Theorem 3.1.23 in~\cite{nesterov2018lectures}):
\beq \label{StatCond}
\ba{rcl}
\la \nabla f(\xx) + \nabla^2 f(\xx)(\xx^+ - \xx) + \beta \mat{B}(\xx^+ - \xx), \yy - \xx^+ \ra
+ \psi(\yy) & \geq & \psi(\xx^+), \qquad \forall \yy \in Q.
\ea
\eeq
Therefore, we have an access to the following special selection
of the subgradient at new point:
\beq \label{FPrimeDef}
\ba{rcl}
F'(\xx^+) & \Def & 
\nabla f(\xx^+)
- \nabla f(\xx) - \nabla^2 f(\xx)(\xx^+ - \xx) - \beta \mat{B}(\xx^+ - \xx)
\;\; \in \;\; \partial F(\xx^+).
\ea
\eeq
Now, utilizing convexity of $\psi$, we claim the following important bound 
for the length of one step. This lemma is the main tool, which is used in 
the core of our analysis.

\begin{lemma} \label{LemmaStepBound} 
 For any $F'(\xx) \in \partial F(\xx)$, we have
\beq \label{StepBound}
\ba{rcl}
\|\xx^+ - \xx\| & \leq & \frac{\| F'(\xx) \|_*}{\beta + \lambda(\xx)},
\ea
\eeq
and
\beq \label{StepBoundX}
\ba{rcl}
\|\xx^+ - \xx\|_{\xx}^2 & \leq & \| \xx^+ - \xx\| \cdot  \| F'(\xx) \|_*
\;\; \overset{\eqref{StepBound}}{\leq} \;\;
\frac{ \| F'(\xx) \|_*^2 }{\beta + \lambda(\xx)}.
\ea
\eeq
\end{lemma}
\begin{proof}
Indeed, for $\psi'(\xx) := F'(\xx) - \nabla f(\xx) \in \partial \psi(\xx)$ and 
$\psi'(\xx^+) := F'(\xx^+) - \nabla f(\xx^+) \in \partial \psi(\xx^+)$, we obtain
$$
\ba{rcl}
0 & \leq & \la \psi'(\xx^+) - \psi'(\xx), \xx^+ - \xx \ra
\;\; \overset{\eqref{FPrimeDef}}{=} \;\;
\la F'(\xx), \xx - \xx^+ \ra - \|\xx^+ - \xx\|_{\xx}^2 - \beta \|\xx^+ - \xx\|^2.
\ea
$$	
Therefore, rearranging the terms and using Cauchy-Schwartz inequality, we get
$$
\ba{rcl}
\|\xx^+ - \xx\| \cdot \|F'(\xx) \|_*
& \geq & 
 \beta \|\xx^+ - \xx\|^2 + \|\xx^+ - \xx\|_{\xx}^2
\;\; \geq \;\;
(\beta + \lambda(\xx)) \|\xx^+ - \xx\|^2.
\ea
$$
which gives the required bounds.
\end{proof}

For our purposes, the most important
inequality is \eqref{StepBound}.
It shows that using the 
\textit{gradient regularization} rule \eqref{BetaChoice}
with sufficiently big $\sigma \geq M$,
we can control the multiplicative error terms in \eqref{GradBound},\eqref{FuncBound}, as follows:
\beq \label{Rho}
\ba{rcl}
\varphi\bigl( M\|\xx^+ - \xx\| \bigr) & \overset{\eqref{StepBound}}{\leq} &
\varphi\Bigl( \frac{M \|F'(\xx) \|^* }{\beta} \Bigr)
\;\; \overset{\eqref{BetaChoice}}{=} \;\;
\varphi\bigl( \frac{M}{\sigma}  \bigr)
\;\; \leq \;\; \varphi(1) \;\; \Def \;\; \rho,
\ea
\eeq
where we used monotonicity of $\varphi$. Note that
\beq \label{RhoValue}
\ba{rcl}
\rho & = & e - 2 \;\; \approx \;\; 0.718.
\ea
\eeq
We are ready to characterize one
 primal Newton step \eqref{PrimalStep}
 with Gradient Regularization rule \eqref{BetaChoice}.
 
 \begin{theorem} \label{TheoremOneStep}
 	Let $\sigma \geq M$.
 	Then, for step \eqref{PrimalStep}
 	with $\beta := \sigma \|F'(\xx) \|_*$ with 
 	some $F'(\xx) \in \partial F(\xx)$,
 	it holds
 	\beq \label{NewGradProgress}
 	\ba{rcl}
 	\la F'(\xx^+), \xx - \xx^+ \ra & \geq & \frac{1}{2\beta} \|F'(\xx^+) \|_*^2.
 	\ea
 	\eeq
 \end{theorem}
 \begin{proof}
 	From~\eqref{GradBound}, we have
 	\beq \label{NewGradEuclidean}
 	\ba{rcl}
 	\| F'(\xx^+) + \beta \mat{B}(\xx^+ - \xx) \|_*
 	& \overset{\eqref{FPrimeDef}}{=} &
 	\| \nabla f(\xx^+) - \nabla f(\xx) - \nabla^2 f(\xx)(\xx^+ - \xx) \|_* \\
 	\\
 	& \overset{\eqref{GradBound}}{\leq} &
 	M\|\xx^+ - \xx\|_{\xx}^2 \cdot \varphi\bigl( M\|\xx^+ - \xx\| \bigr) 
 	\;\; \overset{\eqref{Rho}}{\leq} \;\;
 	\rho M \|\xx^+ - \xx\|_{\xx}^2 \\
 	\\
 	& \overset{\eqref{StepBoundX}}{\leq} &
 	\rho M \| \xx^+ - \xx \| \cdot \|F'(\xx) \|_*
 	\;\; \overset{\eqref{BetaChoice}}{=} \;\;
 	\frac{\rho M}{\sigma} \beta  \| \xx^+ - \xx \|.
 	\ea
 	\eeq
 	Thus, taking square of both sides and rearranging the terms, we obtain
 	$$
 	\ba{rcl}
 	\la F'(\xx^+), \xx - \xx^+ \ra & \overset{\eqref{NewGradEuclidean}}{\geq} &
 	\frac{1}{2\beta} \|F'(\xx^+) \|_*^2
 	\; + \; \frac{\beta}{2} \| \xx^+ - \xx \|^2 \cdot \Bigl(  
 	1 - \bigl[  \frac{\rho M}{\sigma} \bigr]^2 
 	 \Bigr) \\
 	 \\
 	 & \geq & \frac{1}{2\beta} \|F'(\xx^+) \|_*^2,
 	\ea
 	$$
 	where the last bound holds due to our choice: $\sigma \geq M$.
 \end{proof}
 
 From \eqref{NewGradProgress}, we see that one step of our method is \underline{monotone}:
 $F(\xx^+) \leq F(\xx)$. Indeed, by convexity
 \beq \label{FuncGlobal}
 \ba{rcl}
 F(\xx) - F(\xx^+) & \geq & \la F'(\xx^+), \xx - \xx^+ \ra
 \;\; \overset{\eqref{NewGradProgress}}{\geq} \;\;
 \frac{1}{2\beta} \|F'(\xx^+)\|_*^2,
\ea
 \eeq
and for each step we have the global progress
in terms of the objective function value.
 
Let us present our primal method
in the algorithmic form.

\beq\label{PrimalNewton}
\ba{|c|}
\hline\\
\quad \mbox{\bf Newton Method with Gradient Regularization} \quad\\
\\
\hline\\
\ba{l}
\mbox{{\bf Choose} $\xx_0 \in Q$ and 
	set
	$g_0 = \| F'(\xx_0) \|_{*}$ for $F'(\xx_0) \in \partial F(\xx_0)$.} \\
\\
\mbox{\bf For $k \geq 0$ iterate:}\\[10pt]
\mbox{1. For some $\sigma_k \geq 0$, compute } \\[10pt]
	$$
	\ba{rcl}
	\xx_{k + 1}  & =  &
	\arg\min\limits_{\yy \in Q} \Bigl\{  
	\la \nabla f(\xx_k), \yy - \xx_k \ra + \frac{1}{2}\|\yy - \xx_k\|_{\xx_k}^2
	+ \frac{\sigma_k g_k}{2}\|\yy - \xx_k\|^2 + \psi(\yy)  \Bigr\}.
	\ea
	$$ \\
\\
\mbox{2. Set $g_{k + 1} = \|  \nabla f(\xx_{k + 1}) - \nabla f(\xx_k) - 
	(\nabla^2 f(\xx_k) + \sigma_k g_k \mat{B}) (\xx_{k + 1} - \xx_k)  \|_*$.} \\
\\
\ea\\
\hline
\ea
\eeq

Note that in this algorithm we have a freedom of choosing
the regularization parameter $\sigma_k \geq 0$ for each iteration $k \geq 0$.
A simple choice that follows directly from our theory, is the constant one:
$$
\ba{rcl}
\sigma_k & \equiv & M, \qquad \forall k,
\ea
$$
where $M$ is the parameter of Quasi-Self-Concordance \eqref{QSCDef}.
However, if constant $M$ is unknown,
we can do a simple adaptive search that ensures sufficient progress~\eqref{FuncGlobal} 
for the objective function value.
Namely, in each iteration $k \geq 0$, we choose $\sigma_k > 0$ such that
the following inequality is satisfied:
\beq \label{OneStepProgress}
\ba{rcl}
\la F'(\xx_{k + 1}), \xx_k - \xx_{k + 1} \ra
 & \geq & 
\frac{g_{k + 1}^2}{2\sigma_k g_k} ,
\ea
\eeq
by increasing $\sigma_k$ twice.
From Theorem~\ref{TheoremOneStep}, we know that it will be satisfied
at least for $\sigma_k \geq M$.
Then, after accepting a certain $\sigma_k$ for the current iteration, we 
start our adaptive search in the next step from the value
$\sigma_{k + 1} := \frac{\sigma_k}{2}$. This allows both increase and decrease
of the regularization parameters.
It is easy to show that the price of such adaptive search is just \textit{one} extra step per iteration in average.

Note that our algorithm~\eqref{PrimalNewton} with implementation of this adaptive procedure
and using the progress condition \eqref{OneStepProgress}
actually coincides with the recently proposed \textit{Super-Universal Newton Method} 
(Algorithm 2 in \cite{doikov2022super} with parameter $\alpha := 1$).
Therefore, our new theory also extends the analysis
of the Super-Universal Newton Method~\cite{doikov2022super} for the Quasi-Self-Concordant functions.

Let us justify the global linear rate for our method.
For the analysis of our primal scheme, we introduce the initial sublevel set:
$
\mathcal{L}_0  \Def 
\bigl\{ \xx \in Q \, : \, F(\xx) \leq F(\xx_0)  \bigr\},
$
and its diameter, which we assume to be bounded:
\beq \label{DefD}
\ba{rcl}
D & \Def & \sup\limits_{\xx, \yy \in \mathcal{L}_0} \|\xx - \yy\|
\;\; < \;\; +\infty.
\ea
\eeq

We prove our main result.
\begin{theorem} \label{TheoremPrimalRate}
	Let the smooth component of problem~\eqref{MainProblem}
	be Quasi-Self-Concordant with parameter $M \geq 0$.
	Let in each iteration of method~\eqref{PrimalNewton},
	we choose
	parameter $\sigma_k \in [0, 2M]$ such that condition \eqref{OneStepProgress} holds.\footnote{For example, the constant choice $\sigma_k \equiv M$.}
	Then, we have the \underline{global linear rate}, for $k \geq 1$:
	\beq \label{PrimalRate}
	\ba{rcl}
	F(\xx_k) - F^{\star} & \leq & 
	\exp\Bigl(  - \frac{k}{8MD}  \Bigr) \bigl( F(\xx_0) - F^{\star}  \bigr)
	\, + \, 
	\exp\Bigl(  -\frac{k}{4} \Bigr) g_0 D.
	\ea	
	\eeq
\end{theorem}
\begin{proof}
	
We denote $F_k \Def F(\xx_k) - F^{\star}$.
Due to convexity of $F(\cdot)$
and monotonicity of the sequence $\{ F_k \}_{k \geq 0}$, we have
\beq \label{FConv}
\ba{rcl}
F_k & \leq & g_k \|\xx_k - \xx^{\star} \| \;\; \leq \;\; g_k D.
\ea
\eeq
By concavity of logarithm, we know that
\beq \label{LogConcave}
\ba{rcl}
\ln\frac{a}{b} & = & \ln a - \ln b 
\;\; \geq \;\; \frac{1}{a}(a - b), \qquad a, b > 0.
\ea
\eeq
Then, for any $0 \leq i < k$, we obtain
\beq \label{PreTelescopeProgress}
\ba{rcl}
\ln \frac{F_i}{F_{i + 1}} & \overset{\eqref{LogConcave}}{\geq} &
\frac{1}{F_i}( F_i - F_{i + 1} )
\;\; \overset{\eqref{OneStepProgress}}{\geq} \;\;
\frac{g_i}{2 \sigma_i F_i} \cdot \bigl[ \frac{g_{i + 1}}{g_i} \bigr]^2  \\
\\
& \overset{\eqref{FConv}}{\geq} &
\frac{1}{2\sigma_i D} \cdot \bigl[ \frac{g_{i + 1}}{g_i} \bigr]^2 
\;\; \geq \;\;
\frac{1}{4MD} \cdot \bigl[ \frac{g_{i + 1}}{g_i} \bigr]^2.
\ea
\eeq
Therefore, telescoping this bound for all iterations $0 \leq i < k$
and using the inequality between arithmetic and geometric means, we get
\beq \label{FkGlobal}
\ba{rcl}
\ln \frac{F_0}{F_k}
& \overset{\eqref{PreTelescopeProgress}}{\geq} & 
\frac{1}{4 M D} \sum\limits_{i = 0}^{k - 1}
\bigl[ \frac{g_{i + 1}}{g_i} \bigr]^2 
\;\; \geq \;\;
\frac{k}{4 M D} 
\biggl[ \,\prod\limits_{i = 0}^{k - 1} \frac{g_{i + 1}}{g_i} \,\biggr]^{2/k} 
\;\; = \;\;
\frac{k}{4 M D} \bigl[ \frac{g_k}{g_0}  \bigr]^{2/k} \\
\\
& = &
\frac{k}{4 M D} 
\exp\Bigl[ \frac{2}{k} \ln\frac{g_k}{g_0}  \Bigr]
\;\; \geq \;\;
\frac{k}{4 M D} \Bigl( 1 + \frac{2}{k} \ln \frac{g_k}{g_0}  \Bigr) 
\;\; \overset{\eqref{FConv}}{\geq} \;\;
\frac{k}{4 M D} \Bigl( 1 + \frac{2}{k} \ln \frac{F_k}{g_0 D}  \Bigr).
\ea
\eeq
Let us consider two cases.
\begin{enumerate} 
	\item We have $\frac{2}{k} \ln \frac{F_k}{g_0 D} \leq - \frac{1}{2}
\;\; \Leftrightarrow \;\; F_k \leq \exp( -k / 4) g_0D$.

	\item Otherwise, $\frac{2}{k} \ln \frac{F_k}{g_0 D} \geq - \frac{1}{2}$,
	substituting which into \eqref{FkGlobal} gives
	$$
	\ba{rcl}
	\ln \frac{F_0}{F_k} & \geq & \frac{k}{8MD}
	\quad \Leftrightarrow \quad
	F_k \;\; \leq \;\; \exp\Bigl(  -\frac{k}{8MD} \Bigr) F_0.
	\ea
	$$
\end{enumerate}
Combining these two cases, we obtain \eqref{PrimalRate}.
\end{proof}

Since our problem class is \textit{scale-invariant},
it is a natural goal to ensure the \textit{relative accuracy} guarantee.
Thus, we are interested to find a point $\xx_k$ such that
\beq \label{RelAccuracyGuarantee}
\ba{rcl}
F(\xx_k) - F^{\star} & \leq & \varepsilon  \cdot \bigl( F(\xx_0) - F^{\star}  \bigr),
\ea
\eeq
for a certain desired $0 < \varepsilon < 1$.
Note that \eqref{RelAccuracyGuarantee}
can be easily translated into \textit{absolute accuracy} for a given tolerance level $\varepsilon'$ by setting $\varepsilon := \frac{\varepsilon'}{F(\xx_0) - F^{\star}}$.
At the same time, our accuracy parameter $\varepsilon$ in \eqref{RelAccuracyGuarantee} 
does not depend on any particular scaling of the objective function,
which can be arbitrary in practice.
Then, directly from Theorem~\ref{TheoremPrimalRate}, we obtain the following
global complexity guarantee.

\begin{corollary}
	The number of iterations of method~\eqref{PrimalNewton}
	 to reach $\varepsilon$-accuracy 
	is bounded as
	\beq \label{PrimalNewtonComplexity}
	\ba{rcl}
	k & \leq & \cO\Bigl(  MD \ln \frac{1}{\varepsilon} + \ln \frac{Dg_0}{\varepsilon (F(\xx_0) - F^{\star})}  \Bigr).
	\ea
	\eeq
\end{corollary}

We see that the main complexity factor for solving problem \eqref{MainProblem}
with a Quasi-Self-Concordant objective
is the following condition number 
$ \kappa \;\Def \; MD$. Note that this is the same
complexity as we can prove for the trust-region scheme on Hessian-stable functions (see the discussion in the end of previous section).
However, our method is much easier to implement,
using only the basic steps of the form~\eqref{NewtonStep}. 
Moreover, our algorithm admits using a simple adaptive search, which
does not need to know the actual value of parameter $M$.
 
In Sections~\ref{SectionDual} and \ref{SectionAccelerated},
we show how to improve this factor
with the Dual and Accelerated Newton schemes.
The latter method achieves the accelerated complexity $\tilde{\cO}(\kappa^{2/3})$.
The key to our analysis is the \textit{local behaviour}
of our primal Newton method~\eqref{PrimalNewton}, which we 
study in the next section.

\section{Local Quadratic Convergence}
\label{SectionLocal}

Let us assume that $\lambda(\xx_0) > 0$.
Then, we have $\lambda(\xx) > 0$ for all $\xx \in \dom f$
(see Remark~\ref{HessPositiveDef}). 
It is convenient to consider the following scale-invariant measure of convergence:
$$
\ba{rcl}
\eta(\xx) & \Def & \frac{\| F'(\xx) \|_*}{\lambda(\xx)},
\qquad \xx \in Q,
\ea
$$
for a given particular selection of subgradients $F'(\xx) \in \partial F(\xx)$.
We also introduce the following \textit{local region}, for some $0 < c \leq 1$:
\beq \label{LocalRegion}
\ba{rcl}
\mathcal{G}_{c} & \Def & 
\Bigl\{  \xx \in Q \; : \;  \eta(\xx) \leq \frac{c}{M}   \Bigr\}.
\ea
\eeq
Then, for one primal step $\xx \mapsto \xx^+$
of our method~\eqref{PrimalStep} from $\xx \in \mathcal{G}_c$,
and using the Gradient Regularization $\beta = \sigma \| F'(\xx) \|_*$ with
an \textit{arbitrary} $\sigma \geq 0$, we get
\beq \label{LocalStep}
\ba{rcl}
\| F'(\xx^+) \|_* & \overset{\eqref{FPrimeDef}, \eqref{GradBound}}{\leq} &
M \| \xx^+ - \xx\|_{\xx}^2
\cdot \varphi( M\|\xx^+ - \xx \| )
+ \beta \| \xx^+ - \xx \| \\
\\
& \overset{\eqref{StepBound}, \eqref{StepBoundX}}{\leq} &
\frac{M \|F'(\xx)\|^2_*}{\lambda(\xx)}
\cdot \varphi( M \eta(\xx) )
+ \frac{\sigma \|F'(\xx) \|_*^2}{\lambda(\xx)} 
\;\; \overset{\eqref{LocalRegion}}{\leq} \;\;
\frac{\|F'(\xx)\|_*^2}{\lambda(\xx)} \cdot ( \rho M + \sigma ),
\ea
\eeq
where $\rho \Def \varphi(1)$ is an absolute constant \eqref{RhoValue}.
Dividing the left hand side of \eqref{LocalStep} by $\lambda(\xx^+)$, we obtain
\beq \label{LocalMeasureStep}
\ba{rcl}
\eta(\xx^+)
&\overset{\eqref{HessBound}}{\leq} &
\frac{\|F'(\xx^+)\|_*}{\lambda(\xx)} \cdot e^{M\|\xx^+ - \xx\|}
\;\; \overset{\eqref{StepBound}}{\leq} \;\;
\frac{\|F'(\xx^+)\|_*}{\lambda(\xx)} \cdot e^{M \eta(\xx)} \\
\\
& \overset{\eqref{LocalRegion}}{\leq} &
\frac{\|F'(\xx^+)\|_*}{\lambda(\xx)} \cdot e
\;\; \overset{\eqref{LocalStep}}{\leq} \;\;
\eta(\xx)^2 \cdot e( \rho M + \sigma ),
\ea
\eeq
which is the quadratic convergence in terms of our measure $\eta(\xx)$.
Thus, we established the following.

\begin{theorem} \label{TheoremLocalQuadratic}
Let the smooth component of problem \eqref{MainProblem}
be Quasi-Self-Concordant with parameter $M > 0$.
Consider the iterations $\{ \xx_k \}_{k \geq 0}$ of method \eqref{PrimalNewton}
with any $\sigma_k \in [0, 2M]$,
and assume that $\xx_0 \in \mathcal{G}_{1/18}$.
Then, we have the quadratic rate, for $k \geq 0$:
\beq \label{LocalQuadratic}
\ba{rcl}
\eta(\xx_k) & \leq & \frac{1}{9M} \cdot \bigl( \frac{1}{2} \bigr)^{2^k}.
\ea
\eeq
\end{theorem}
\begin{proof}
Indeed, for any $k \geq 1$, we have 
$\eta(\xx_{k}) \overset{\eqref{LocalMeasureStep}}{\leq} \eta(\xx_{k - 1})^2 \cdot e(\rho M + \sigma)
\leq \eta(\xx_{k - 1})^2 \cdot 9M$.
Therefore, by induction $\xx_k \in \mathcal{G}_{1/18}$ for all $k \geq 0$, and we get
$$
\ba{rcl}
( \eta(\xx_k) \cdot 9M  ) & \leq & (\eta(\xx_{k - 1}) \cdot 9M  )^2
\;\; \leq \;\; \ldots \;\; \leq \;\;
( \eta(\xx_0) \cdot 9M )^{2^k}.
\ea
$$
Taking into account that $\xx_0 \in \mathcal{G}_{1/18}$ finishes the proof.
\end{proof}

We see that for achieving the local quadratic convergence \eqref{LocalQuadratic}, we can
perform just \textit{pure} Newton steps ($\sigma_k \equiv 0$). 

Let us consider a particular consequence of our theory for minimizing strongly convex functions, 
that will be useful in the next section
for constructing the Dual Newton Method. 
Thus, we assume that the Hessian is uniformly bounded from below, for some $\mu > 0$:
\beq \label{StronglyConvex}
\ba{rcl}
\lambda(\xx) & \geq & \mu, \qquad \forall \xx \in Q,
\ea
\eeq
and consider one primal \textit{pure} Newton step \eqref{PrimalStep} $\xx \mapsto \xx^+$ with $\sigma := 0$ (no regularization).
Then, assuming that point $\xx$ belongs to the following local region,
for some selection of the subgradients:
\beq \label{Local2}
\ba{rcl}
\mathcal{U} & \Def &
\Bigl\{ \xx \in Q \; : \;  \| F'(\xx) \|_*  \leq \frac{\mu}{2M}  \Bigr\},
\ea
\eeq
we get
\beq \label{LocalStep2}
\ba{rcl}
\|F'(\xx^+) \|_*
& \overset{\eqref{FPrimeDef}, \eqref{GradBound}}{\leq} &
M \| \xx^+ - \xx\|_{\xx}^2
\cdot \varphi( M\|\xx^+ - \xx \| ) \\
\\
& \overset{\eqref{StepBound}, \eqref{StepBoundX},\eqref{StronglyConvex}}{\leq}  &
\frac{M \|F'(\xx) \|_*^2}{\mu} \cdot \varphi\Bigl( \frac{M \|F'(\xx)\|_*}{\mu} \Bigr)
\;\; \overset{\eqref{Local2}}{\leq} \;\;
\frac{M \|F'(\xx) \|_*^2}{\mu} \cdot \varphi(1/2).
\ea
\eeq
Thus, we can justify the following quadratic convergence
for our specific case. 

\begin{theorem} \label{TheoremLocal2}
	Let the smooth component of problem \eqref{MainProblem}
	be Quasi-Self-Concordant with parameter $M > 0$
	and strongly convex  with parameter $\mu > 0$.
	Consider the iterations of the pure Newton Method:
	$$
	\ba{rcl}
	\xx_{k + 1} & = & \arg\min\limits_{\yy \in Q}
	\Bigl\{ 
	\la \nabla f(\xx_k), \yy - \xx_k \ra + \frac{1}{2}\|\yy - \xx_k\|_{\xx_k}^2 + \psi(\yy)
	\Bigr\}, \qquad k \geq 0,
	\ea
	$$
	and assume that $\xx_0 \in \mathcal{U}$. Then, we have the quadratic rate, for $k \geq 0$:
	\beq \label{LocalQuadraticRate}
	\ba{rcl}
	\| F'(\xx_k) \|_* & \leq & \frac{\mu}{2M} \cdot \bigl( \frac{1}{e} \bigr)^{2^k}.
	\ea
	\eeq
\end{theorem}
\begin{proof}
	Indeed, for any $k \geq 1$, we have
	$$
	\ba{rcl}
	\frac{M \varphi(1/2) }{\mu} \|F'(\xx_k) \|_*
	& \overset{\eqref{LocalStep2}}{\leq} &
	\Bigl(  \frac{M \varphi(1/2) }{\mu} \|F'(\xx_{k - 1}) \|_* \Bigr)^2
	\;\; \leq \;\; \ldots \;\; \leq \;\;
	\Bigl(  \frac{M \varphi(1/2) }{\mu} \|F'(\xx_{0}) \|_* \Bigr)^{2^k}.
	\ea
	$$
	It remains to notice that $1/2 \leq \varphi(1/2) \leq 2 / e$.
	Using that $\xx_0 \in \mathcal{U}$ completes the proof.
\end{proof}

\section{Dual Newton Method}
\label{SectionDual}

In this section, we discover another possibility for minimizing our objective \eqref{MainProblem}
which is based on inexact \textit{proximal-point}  iterations
\cite{moreau1965proximite,martinet1978perturbation,rockafellar1976augmented,solodov2001unified,salzo2012inexact}.
Thus, at each iteration $k \geq 0$, we 
 minimize the following augmented function:
\beq \label{ProxH}
\ba{rcl}
h_{k} (\xx) & := & F(\xx) + \frac{1}{2a_{k + 1}}\|\xx - \xx_k \|^2,
\ea
\eeq
where $\xx_k$ is the current point and $a_{k + 1} > 0$ is some regularization constant.
Note that the smooth component of \eqref{ProxH}
is Quasi-Self-Concordant with parameter $M$ and 
strongly convex with constant $\mu = \frac{1}{a_{k + 1}}$. Hence, choosing
$$
\boxed{
\ba{rcl}
a_{k + 1} & := & \frac{1}{2 M \| F'(\xx_k) \|_*} 
\ea
}
$$
we ensure that point $\xx_k$ is in the \textit{local region of quadratic convergence} \eqref{Local2}
of the Newton Method, and we can minimize objective \eqref{ProxH} very fast
up to any desirable precision, to obtain the next iterate $\xx_{k + 1} \approx \arg\min\limits_{\xx} h_k(\xx)$. 
This approach 
resembles the standard Path-Following scheme for minimizing the classic Self-Concordant functions \cite{nesterov2018lectures}.
It was also explored recently in \cite{doikov2021local}
as an application of high-order Tensor Methods.

\beq\label{DualNewton}
\ba{|c|}
\hline\\
\quad \mbox{\bf Dual Newton Method $\mathscr{D}(f, \psi, \xx_0, M, \nu)$} \quad\\
\\
\hline\\
\ba{l}
\mbox{\textbf{Initialize}
	$g_0 = \| F'(\xx_0) \|_{*}$ for $F'(\xx_0) \in \partial F(\xx_0)$.} \\[10pt]
\mbox{\bf For $k \geq 0$ iterate:}\\[10pt]
\mbox{1. Set $\zz_0 = \xx_k$.} \\[10pt]
\mbox{2. \textbf{For $t \geq 0$ iterate:}} \\[10pt]
\mbox{$\quad\;$2-a. Compute} \\[5pt]
$$
\ba{rcl}
\qquad \zz_{t + 1} & = & \arg\min\limits_{\yy \in Q}
\Bigl\{  
\la \nabla f(\zz_t), \yy - \zz_t \ra
+ \frac{1}{2} \| \yy - \zz_t \|_{\zz_t}^2
+ Mg_k \|\yy - \xx_k \|^2 + \psi(\yy)
\Bigr\}.
\ea
$$ \\[10pt]
\mbox{$\quad\;$2-b. Set $\ss_{t + 1} = 
	\nabla f(\zz_{t + 1}) - \nabla f(\zz_t) 
	- \nabla^2 f(\zz_t)(\zz_{t + 1} - \zz_t).$ } \\[10pt]
\mbox{$\quad\;$\textbf{Until $\| \ss_{t + 1}\|_* \leq \frac{2M g_k \nu}{(k + 1)^2}.$}} \\[10pt]
\mbox{3. Set $\xx_{k + 1} = \zz_{t + 1}$ and $g_{k + 1} = \| \ss_{t + 1} - 2Mg_k B(\xx_{k + 1} - \xx_{k})\|_*$.} \\[10pt]
\mbox{4. If $g_{k + 1} \leq \nu$ then \textbf{return} $\xx_{k + 1}$.} \\
\\
\ea\\
\hline
\ea
\eeq

In this algorithm, we need to pass the parameter of Quasi-Self-Concordance $M$
for the smooth component of the objective.
The method returns
a point $\bar{\xx} \in Q$ with a desired bound for the (sub)gradient norm: $\| F'(\bar{\xx}) \|_* \leq \nu$, where $\nu > 0$ is a given tolerance.

Note that the first iteration $\zz_0 \mapsto \zz_1$ of the inner loop in \eqref{DualNewton}
is equivalent to one step of the primal method.
However, we do several iterations of the inner loop until we find
a good minimizer of the augmented objective,
that has a small gradient norm:
$
\| h'(\xx_{k + 1}) \|_* \leq \frac{2Mg_k \nu}{(k + 1)^2}.
$

It appears that the resulting Dual Newton Method
is very suitable for minimizing the dual objects (the gradients)
and importantly it provides us with a \textit{hot-start} possibility, which we use
extensively
for building our accelerated method in the next section.
We establish the following global guarantee.

\begin{lemma} \label{LemmaDualNewton}
	For the iterations of method \eqref{DualNewton}, we have
	\beq \label{DualNewtonGuarantee}
	\ba{rcl}
	\sum\limits_{i = 1}^k a_i \bigl(F(\xx_i) - F^{\star} \bigr)
	+ \frac{1}{2}\sum\limits_{i = 1}^k a_i^2 \| F'(\xx_i)\|_*^2
	& \leq & \frac{1}{2}(\|\xx_0 - \xx^{\star}\| + 2\nu  )^2,
	\qquad k \geq 1.
	\ea
	\eeq
\end{lemma}
\begin{proof}
	Let us prove by induction the following bound, for all $k \geq 0$:
	\beq \label{DualInduct}
	\ba{rcl}
	\frac{1}{2}\|\xx_0 - \xx^{\star}\|^2
	+ \sum\limits_{i = 1}^k a_i F^{\star} 
	+ C_k
	& \geq &	 
	\frac{1}{2}\| \xx_k - \xx^{\star}\|^2
	+ \sum\limits_{i = 1}^k a_i F(\xx_i)
	+ \frac{1}{2}\sum\limits_{i = 1}^k a_i^2 \|F'(\xx_i)\|_*^2,
	\ea
	\eeq
	where
	$$
	\ba{rcl}
	C_k & \Def & 
	\frac{1}{2} \sum\limits_{i = 1}^k \xi_i^2 + \sum\limits_{i = 1}^k \xi_i \|\xx_{i - 1} - x^{\star} \|,
	\qquad 
	\xi_i \;\; \Def \;\; \frac{\eta}{i^2}.
	\ea
	$$
	Inequality \eqref{DualInduct} obviously holds for $k = 0$. Assume that it holds for some $k \geq 0$ and consider
	one iterations of the method. We have
	\beq \label{DualOneStep}
	\ba{cl}
	& \frac{1}{2}\|\xx_0 - \xx^{\star}\|^2 + \sum\limits_{i = 1}^{k + 1} a_i F^{\star} \\[10pt]
	& \overset{\eqref{DualInduct}}{\geq} \;\;
	\frac{1}{2}\| \xx_k - \xx^{\star} \|^2 + a_{k + 1} F^{\star} + \sum\limits_{i = 1}^k a_i F(\xx_i)
	+ \frac{1}{2}\sum\limits_{i = 1}^k a_i^2 \|F'(\xx_i) \|_*^2 - C_k.
	\ea
	\eeq
	Note that the stopping condition for the inner loop in Step 2-c ensures that
	\beq \label{DualDeltaBound}
	\ba{rcl}
	\frac{\xi_{k + 1}}{a_{k + 1}} & \geq & \| \ss_{t + 1} \|_*
	\;\; = \;\; \| F'(\xx_{k + 1}) + \frac{1}{a_{k + 1}} B(\xx_{k + 1} - \xx_k) \|_*.
	\ea
	\eeq
	Taking square of both sides of \eqref{DualDeltaBound} and rearranging the terms, we obtain
	\beq \label{DualStepProgress}
	\ba{rcl}
	a_{k + 1} \la F'(\xx_{k + 1}), \xx_k - \xx_{k + 1} \ra & \geq & 
	\frac{a_{k + 1}^2}{2}\| F'(\xx_{k + 1}) \|_*^2
	+ \frac{1}{2}\| \xx_{k} - \xx_{k + 1} \|^2 - \frac{1}{2} \xi_{k + 1}^2
	\ea
	\eeq
	Therefore, we can estimate the first two terms in the right hand side of \eqref{DualOneStep},
	as follows
	$$
	\ba{cl}
	& \frac{1}{2}\|\xx_k - \xx^{\star}\|^2 + a_{k + 1} F^{\star}
	\;\; \geq \;\;
	\frac{1}{2}\| \xx_k - \xx^{\star} \|^2 + a_{k + 1} \bigl[ F(\xx_{k + 1}) + \la F'(\xx_{k + 1}), \xx^{\star} - \xx_{k + 1} \ra \bigr] \\[10pt]
	& \; = \;\,
	\frac{1}{2}\|\xx_{k + 1} - \xx^{\star}\|^2 + a_{k + 1} F(\xx_{k + 1}) + a_{k + 1} \la F'(\xx_{k + 1}), \xx_k - \xx_{k + 1} \ra \\[10pt]
	& \qquad \qquad + \; \frac{1}{2} \|\xx_{k} - \xx_{k + 1} \|^2 + \la B(\xx_k - \xx_{k + 1}), \xx_{k + 1} - \xx^{\star} \ra
	+ a_{k + 1} \la F'(\xx_{k + 1}), \xx^{\star} - \xx_k \ra \\[10pt]
	& \overset{\eqref{DualStepProgress}}{\geq} \;
	\frac{1}{2}\|\xx_{k + 1} - \xx^{\star}\|^2 + a_{k + 1} F(\xx_{k + 1})
	+ \frac{a_{k + 1}^2}{2} \|F'(\xx_{k + 1}) \|_*^2 - \frac{1}{2} \xi_{k + 1}^2 \\[10pt]
	& \qquad \qquad + \;
	\la a_{k + 1} F'(\xx_{k + 1}) + B(\xx_{k + 1} - \xx_k), \xx^{\star} - \xx_k \ra \\[10pt]
	& \overset{\eqref{DualDeltaBound}}{\geq} \;
	\frac{1}{2}\|\xx_{k + 1} - \xx^{\star}\|^2 + a_{k + 1} F(\xx_{k + 1})
	+ \frac{a_{k + 1}^2}{2} \|F'(\xx_{k + 1}) \|_*^2 - \frac{1}{2} \xi_{k + 1}^2 - \xi_k \|\xx_{k} - \xx^{\star}\|.
	\ea
	$$
	Combining it with \eqref{DualOneStep}, we obtain \eqref{DualInduct} for the next iterate.
	Thus, \eqref{DualInduct} is proven for all $k \geq 0$.
	
	Now, let us denote $R_k \Def \|\xx_0 - \xx^{\star}\| + \sum_{i = 1}^k \xi_i$, and prove by induction that
	\beq \label{DualRBound}
	\ba{rcl}
	\frac{1}{2}\|\xx_0 - \xx^{\star}\|^2 + C_k & \leq & \frac{1}{2} R_k^2,
	\ea
	\eeq
	which trivially holds for $k = 0$. Assume that \eqref{DualRBound} holds for some $k \geq 0$. Then, from \eqref{DualInduct}
	we have 
	\beq \label{DualRkBound}
	\ba{rcl}
	\| \xx_k - \xx^{\star}\| & \leq & R_k,
	\ea
	\eeq
	and
	$$
	\ba{rcl}
	\frac{1}{2}\|\xx_0 - \xx^{\star}\|^2 + C_{k + 1} & = & 
	\frac{1}{2}\|\xx_0 - \xx^{\star}\|^2 + C_{k} + \frac{1}{2} \xi_{k + 1}^2 + \xi_{k + 1} \|\xx_k - \xx^{\star} \| \\[10pt]
	& \overset{\eqref{DualRBound},\eqref{DualRkBound}}{\leq} &
	\frac{1}{2} R_k^2 + \frac{1}{2} \xi_{k + 1}^2 + \xi_{k + 1} R_k \;\; = \;\; \frac{1}{2} ( R_{k} + \xi_{k + 1} )^2 
	\;\; = \;\; \frac{1}{2}R_{k + 1}^2,
	\ea
	$$
	which proves \eqref{DualRBound} for all $k \geq 0$. Thus, we justified that
	$$
	\ba{rcl}
	\sum\limits_{i = 1}^k a_{i + 1} ( F(\xx_i) - F^{\star} ) + \frac{1}{2} \sum\limits_{i = 1}^k a_i^2 \| F'(\xx_i) \|^2
	& \overset{\eqref{DualInduct}, \eqref{DualRBound}}{\leq} &  \frac{1}{2} R_k^2.
	\ea
	$$
	It remains to notice that $R_k  \leq  \|\xx_0 - \xx^{\star}\| +  2\nu$, $\forall k \geq 0$.
\end{proof}

From inequality \eqref{DualNewtonGuarantee}, we see that the method has a convergence
both in terms of the functional residual and in terms of the gradient norm.
We prove the following theorem.

\begin{theorem} \label{TheoremDualNewtonConvergence}
	Let the smooth component of problem~\eqref{MainProblem}
	be Quasi-Self-Concordant with parameter $M > 0$.
	Then, for the iterations of method~\eqref{DualNewton}
	with tolerance $\nu > 0$,
	we have the \underline{global linear rate}:
	\beq \label{DualNewtonRate}
	\ba{rcl}
	\| F'(\xx_k) \|_* & \leq & 
	\exp\Bigl( 
	2M^2 ( \|\xx_0 - \xx^{\star} \| + 2\nu )^2 - \frac{k}{2}
	\Bigr) \| F'(\xx_0) \|_*, \qquad k \geq 1.
	\ea
	\eeq
	Moreover, during the first $k$ iterations of the method,
	the total number $N_k$ of second-order oracle calls is bounded as
	\beq \label{DualNewtonTotal}
	\ba{rcl}
	N_k & \leq & k \cdot \Bigl(1 + \frac{1}{\ln2} \ln \ln \frac{(k + 1)^2}{2M\nu} \Bigr).
	\ea
	\eeq
\end{theorem}
\begin{proof}
	Estimate \eqref{DualNewtonRate} follows immediately from \eqref{DualNewtonGuarantee}
	and our choice of parameter $a_{k + 1} = \frac{1}{2M \|F'(\xx_k)\|_*}$.
	Indeed, using the inequality between
	arithmetic and geometric means, we get
	$$
	\ba{rcl}
	\frac{1}{2}( \|\xx_0 - \xx^{\star}\| + 2\nu )^2 
	& \overset{\eqref{DualNewtonGuarantee}}{\geq} &
	\frac{1}{8M^2} \sum\limits_{i = 1}^k \frac{\| F'(\xx_i)\|_*^2}{\| F'(\xx_{i - 1})\|_*^2} 
	\;\; \geq \;\;
	\frac{k}{8M^2} \biggl[ 
	\prod\limits_{i = 1}^k \frac{\|F'(\xx_i)\|_*}{\| F'(\xx_{i - 1})\|_*}  \biggr]^{2 / k} \\
	\\
	& = &
	\frac{k}{8M^2} \Bigl[  \frac{\| F'(\xx_k) \|_*}{\| F'(\xx_0) \|_*} \Bigr]^{2/k}
	\;\; = \;\;
	\frac{k}{8M^2} \exp\Bigl[ \frac{2}{k}\ln \frac{\| F'(\xx_k) \|_*}{\| F'(\xx_0) \|_*} \Bigr] \\
	\\
	& \geq & 
	\frac{k}{8M^2} \Bigl( 1 +  \frac{2}{k}\ln \frac{\| F'(\xx_k) \|_*}{\| F'(\xx_0) \|_*}  \Bigr)
	\;\; = \;\;
	\frac{k}{8M^2} + \frac{1}{4M^2}\ln \frac{\| F'(\xx_k) \|_*}{\| F'(\xx_0) \|_*},
	\ea
	$$
	which is \eqref{DualNewtonRate}.
	
	For iteration $i \geq 0$, let us denote by $t_i \geq 1$
	the corresponding number of second-order oracle calls (executions of Step 2-a of the method).
	Then, due to the local quadratic convergence (Theorem~\ref{TheoremLocal2}), it is enough to ensure
	$$
	\ba{rcl}
	\frac{\mu}{2M} \cdot \bigl( \frac{1}{e} \bigr)^{2^{t_i}}
	\;\; = \;\; \frac{1}{2M a_{k + 1}} \cdot \bigl( \frac{1}{e} \bigr)^{2^{t_i}}
	& \leq & \delta_{i + 1} \;\; = \;\; \frac{\nu}{(i + 1)^2 a_{i + 1}},
	\ea
	$$
	which is satisfied for $t_i \geq \frac{1}{\ln 2} \ln \ln \frac{(i + 1)^2}{\nu}$.
	Therefore, the total number of second-order oracle calls $N_k$
	during the first $k$ iterations is bounded as
	$$
	\ba{rcl}
	N_k & = & \sum\limits_{i = 1}^k t_i
	\;\; \leq \;\; \sum\limits_{i = 1}^k
	\Bigl( 1 + \frac{1}{\ln 2} \ln \ln \frac{(i + 1)^2}{\nu} \Bigr)
	\;\; \leq \;\; k \cdot \Bigl( 1 + \frac{1}{\ln 2} \ln \ln \frac{(k + 1)^2}{\nu}  \Bigr),
	\ea
	$$
	which completes the proof.
\end{proof}

\begin{corollary}
	The number of iterations of method~\eqref{DualNewton}
	to reach $\| F'(\xx_k) \|_* \leq \nu$ is bounded as
	\beq \label{DualNewtonComplexity}
	\ba{rcl}
	k & \leq & 
	\cO\Bigl( 
	(M \|\xx_0 - \xx^{\star}\|)^2 + (M\nu)^2 + \ln \frac{ \|F'(\xx_0)\|_* }{\nu}	
	\Bigr).
	\ea
	\eeq
\end{corollary}

It is important that 
we obtained \textit{explicit distance} from the initial point to the solution $\| \xx_0 - \xx^{\star}\|$
in the right hand side of \eqref{DualNewtonComplexity},
as opposed to the primal method, where we had the diameter of the initial sublevel set $D$
(see \eqref{PrimalNewtonComplexity}).
Although, we have square of the condition number $\kappa^{2}$ in \eqref{DualNewtonComplexity},
the \textit{hot-start} possibility will be very useful for us to build the accelerated scheme
in the next section,
which achieves the improved complexity factor $\kappa^{2/3}$.

A similar observation was also made in \cite{kovalev2022first} 
(for different high-order methods),
where the presence of the explicit distance to the solution helped
to prove the optimal accelerated complexity.

\section{Accelerated Newton Scheme}
\label{SectionAccelerated}

Let us present our accelerated second-order scheme,
which is based on \textit{contracting proximal-point} iterations \cite{doikov2020contracting}.
In our algorithm, we have two sequences of points: $\{ \xx_k \}_{k \geq 0}$
(the main sequence), and $\{ \vv_k \}_{k \geq 0}$ (the proximal sequence).
Let us fix a \textit{contraction parameter} $\gamma \in (0, 1)$ and the following sequences,
for $k \geq 1$:
\beq \label{AkGrowth}
\ba{rcl}
A_k & := &  A_0(1 - \gamma)^{-k} \;\; \geq \;\; A_0 e^{k \gamma},
\qquad \text{and} \qquad
a_k \;\; :=  \;\; A_k - A_{k - 1} \;\; = \;\; \frac{\gamma}{1 - \gamma}A_{k - 1},
\ea
\eeq
where $A_0 > 0$ is some initial constant. Hence, 
we have $\frac{a_k}{A_k} \equiv \gamma$, $\forall k \geq 1$. The exponential
grows of $A_k$ will determine the convergence rate of our method.

At iteration $k \geq 1$ of our accelerated scheme, we minimize the following augmented
objective with \textit{contracted}
smooth part (compare with \eqref{ProxH}):
\beq \label{ContrProxH}
\ba{rcl}
h_{k}(\xx) & := & 
\underbrace{ \textstyle
	A_{k + 1}f\bigl( \gamma \xx + (1 - \gamma) \xx_k  \bigr)}_{\Def \;\ \bar{f_k}(\xx)}
\; + \; 
\underbrace{ \textstyle
	a_{k + 1} \psi(\xx) + \frac{1}{2}\| \xx - \vv_k \|^2}_{\Def \; \bar{\psi}_k(\xx)}.
\ea
\eeq
Even though this objective is strongly convex,
it is difficult to control the region of its local quadratic convergence
due the contraction of the smooth part. Instead, for minimizing \eqref{ContrProxH}, we apply
already developed Dual Newton Method~\eqref{DualNewton}
from the previous section as the inner procedure.

Note that the smooth part $\bar{f}_k(\cdot)$ of \eqref{ContrProxH} is Quasi-Self-Concordant
with parameter $\gamma M$.
Let us denote by $\vv_{k + 1}^{\star} \Def \arg\min_{\xx} h_k(\xx)$
its exact minimum, and $\vv_{k + 1} \approx \vv_{k + 1}^{\star}$
is a point that we employ in our scheme, with the following guarantee,
for some tolerance parameter $\nu_{k + 1} > 0$:
\beq \label{InexactVk1}
\ba{rcl}
\| h_k'(\vv_{k + 1}) \|_* & \leq & \nu_{k + 1}.
\ea
\eeq
Since the function $h_k(\cdot)$ is $1$-strongly convex, we have
$$
\ba{rcl}
\| \vv_{k + 1} - \vv_{k + 1}^{\star} \|^2 & \leq &
\la h_k'(\vv_{k + 1}), \vv_{k + 1} - \vv_{k + 1}^{\star} \ra
\;\; \overset{\eqref{InexactVk1}}{\leq} \;\; 
\nu_{k + 1} \|\vv_{k + 1} - \vv_{k + 1}^{\star} \|.
\ea
$$
Thus, $\| \vv_{k + 1} - \vv_{k + 1}^{\star} \| \leq \nu_{k + 1}$ and
\beq \label{AccSolBound}
\ba{rcl}
\|\vv_k - \vv_{k + 1}^{\star}\|  & \leq & 
\| \vv_k - \vv_{k + 1} \| + \nu_{k + 1}.
\ea
\eeq
Therefore,
minimizing~\eqref{ContrProxH} by the Dual Newton Method
and using $\vv_{k}$ as the initial point,
$$
\boxed{
\ba{rcl}
\vv_{k + 1} & := & \mathscr{D}( \bar{f}_k, \bar{\psi}_k, \vv_k, \gamma M, \nu_{k + 1}),
\ea
}
$$
we reach guarantee \eqref{InexactVk1} with
the following number of iteration of our inner method:
\beq \label{InnerMethodIters}
\ba{rcl}
I_k & \overset{\eqref{DualNewtonComplexity}}{\leq} &
\cO\Bigl( 
\gamma^2 M^2 \| \vv_k - \vv_{k + 1} \|^2
+ \gamma^2 M^2 \nu_{k + 1}^2 + \ln \frac{\| h'(\vv_k) \|_*}{\nu_{k + 1}}
\Bigr),
\ea
\eeq
which we are able to effectively control. 
Thus, we come to our accelerated method.

\beq\label{AcceleratedNewton}
\ba{|c|}
\hline\\
\quad \mbox{\bf Accelerated Newton Scheme} \quad\\
\\
\hline\\
\ba{l}
\mbox{{\bf Choose} $\xx_0 \in Q$, $R > 0$, $A_0 > 0$, and $\gamma \in (0, 1)$. Set $\vv_0 = \xx_0$.} \\
\\
\mbox{\bf For $k \geq 0$ iterate:}\\[10pt]
\mbox{1. Set $A_{k + 1} = \frac{1}{1 - \gamma} A_k$, 
				$a_{k + 1} = \frac{\gamma}{1 - \gamma}A_k$, 
	             and 
                 $\nu_{k + 1} = \frac{R}{(k + 1)^2}$.} \\
\\
\mbox{2. Set up contracted smooth part } \\[10pt]
$$
\ba{rcl}
\qquad \qquad \bar{f}_k(\xx) & = & A_{k + 1} f\bigl( \gamma \xx + (1 - \gamma) \xx_k \bigr)
\ea
$$ \\[10pt]
\mbox{ and augmented composite part }\\[10pt]
$$
\ba{rcl}
\qquad \qquad \bar{\psi}_k(\xx) & = & a_{k + 1} \psi(\xx) + \frac{1}{2} \|\xx - \vv_k \|^2.
\ea
$$ \\
\\
\mbox{3. Run Dual Newton Method:} \\[10pt]
$$
\ba{rcl}
\qquad \qquad \vv_{k + 1} & = & \mathscr{D}( \bar{f}_k, \bar{\psi}_k, 
\vv_k, \gamma M, \nu_{k + 1}).
\ea
$$ \\
\\
\mbox{4. Set $\xx_{k + 1} = \gamma \vv_{k + 1} + (1 - \gamma) \xx_k$.} \\
\\
\ea\\
\hline
\ea
\eeq

We establish the following global
guarantee for this process.

\begin{lemma} \label{LemmaAcceleratedGuarantee}
	For the iterations of method~\eqref{AcceleratedNewton} we have, for $k \geq 1$:
	\beq \label{AccGuarantee}
	\ba{cl}
	& A_k (F(\xx_k) - F^{\star}) + 
	\frac{1}{2}\| \vv_k - \xx^{\star} \|^2
	+ \frac{1}{2} \sum\limits_{i = 1}^k \| \vv_{i} - \vv_{i - 1}  \|^2 \\[10pt]
	& \leq \;\; 
	\frac{1}{2} \Bigl( 
	\|\xx_0 - \xx^{\star} \|
	+ \sqrt{2A_0 ( F(\xx_0) - F^{\star} )} 
	+ 4R \Bigr)^2.
	\ea
	\eeq
\end{lemma}
\begin{proof}
	Let us prove by induction the following bound,
	for all $k \geq 0$:
	\beq \label{AccInductBound}
	\ba{rcl}
	\frac{1}{2}R_k^2
	+ A_k F^{\star}
	& \geq & 
	\frac{1}{2}\| \vv_k - \xx^{\star}\|^2
	+
 	 A_k F(\xx_k) + \frac{1}{2}\sum\limits_{i = 1}^k \| \vv_{i} - \vv_{i - 1} \|^2,
	\ea
	\eeq
	where
	$$
	\ba{rcl}
	\frac{1}{2}R_k^2 & \Def & 
	\frac{1}{2}\| \xx_0 - \xx^{\star} \|^2
	+ A_0(F(\xx_0) - F^{\star}) 
	+ \sum\limits_{i = 1}^k
		\nu_i  \|\vv_i - \xx^{\star}\|.
	\ea
	$$
	Inequality~\eqref{AccInductBound} obviously holds for $k = 0$.
	Assume that it is satisfied for some $k \geq 0$ and consider one iteration of the method.
	We have
	\beq \label{AccProve1}
	\ba{cl}
	& \frac{1}{2} R_{k + 1}^2 + A_{k + 1} F^{\star}
	\;\; = \;\;
	\frac{1}{2} R_k^2 + \nu_{k + 1}\| \vv_{k + 1} - \xx^{\star} \| + A_k F^{\star} 
	+ a_{k + 1} F^{\star} \\[10pt]
	& \overset{\eqref{AccInductBound}}{\geq} \;
	\frac{1}{2} \| \vv_k - \xx^{\star}\|^2
	+ A_k F(\xx_k) + \frac{1}{2}\sum\limits_{i = 1}^k \| \vv_{i} - \vv_{i - 1}  \|^2
	+
	a_{k + 1} F^{\star}
	+ \nu_{k + 1}\| \vv_{k + 1} - \xx^{\star} \| \\[10pt]
	& \; \geq \;\;
	h_k(\xx^{\star}) + A_k \psi(\xx_k) + \frac{1}{2}\sum\limits_{i = 1}^k \| \vv_{i} - \vv_{i - 1} \|^2
	+ \nu_{k + 1}\| \vv_{k + 1} - \xx^{\star} \|,
 	\ea
	\eeq
	where we used convexity of $f(\cdot)$ in the last inequality,
	definition \eqref{ContrProxH}, and our choice
	of parameters: $\frac{a_{k + 1}}{A_{k + 1}} \equiv \gamma$.
	Note that function $h_k(\cdot)$ is $1$-strongly convex.
	Hence,
	$$
	\ba{rcl}
	h_k(\xx^{\star}) & \geq & h_k (\vv_{k + 1})
	+ \la h_k'(\vv_{k + 1}), \xx^{\star} - \vv_{k + 1} \ra
	+ \frac{1}{2} \| \xx^{\star} - \vv_{k + 1} \|^2 \\[10pt]
	& \overset{\eqref{InexactVk1}}{\geq} &
	h_k (\vv_{k + 1}) - \nu_{k + 1} \|\vv_{k + 1} - \xx^{\star}\|
	+ \frac{1}{2} \|  \vv_{k + 1} - \xx^{\star} \|^2 \\[10pt]
	& = &
	A_{k + 1} f(\xx_{k + 1}) + a_{k + 1} \psi(\vv_{k + 1})
	+ \frac{1}{2}\| \vv_{k + 1} - \vv_k \|^2  \\[10pt]
	& &
	\qquad \qquad
	- \; \nu_{k + 1} \|\vv_{k + 1} - \xx^{\star}\|
	+ \frac{1}{2} \| \vv_{k + 1} - \xx^{\star} \|^2.
 	\ea
	$$
	Substituting this bound into \eqref{AccProve1}
	and using convexity of $\psi(\cdot)$, we obtain
	inequality~\eqref{AccInductBound}
	for the next iteration.
	Therefore, \eqref{AccInductBound} is proven for all $k \geq 0$.
	
	Now, from \eqref{AccInductBound}, we have that
	\beq \label{AccProxBound}
	\ba{rcl}
	\| \vv_k - \xx^{\star} \| & \leq & R_k, \qquad \forall k \geq 0.
	\ea
	\eeq
	Therefore,
	$$
	\ba{rcl}
	\frac{1}{2}R_{k + 1}^2
	& = & 
	\frac{1}{2} R_k^2 + \nu_{k + 1} \| \vv_{k + 1} - \xx^{\star} \|
	\;\; \overset{\eqref{AccProxBound}}{\leq} \;\;
	\frac{1}{2} R_k^2 + \nu_{k + 1} R_{k + 1}.
	\ea
	$$
	Dividing this inequality by $\frac{1}{2}R_{k + 1}$, we get
	$$
	\ba{rcl}
	R_{k + 1} & \leq & \frac{R_k^2}{R_{k + 1}} + 2\nu_{k + 1}
	\;\; \leq \;\;
	R_k + 2\nu_{k + 1}, \qquad \forall k \geq 0.
	\ea
	$$
	Hence, by induction, we justify that
	$$
	\ba{rcl}
	R_k & \leq & R_0 + 2 \sum\limits_{i = 1}^k \nu_i
	\;\; = \;\; \sqrt{ \|\xx_0 - \xx^{\star}\|^2 + 2A_0(F(\xx_0) - F^{\star}) }
	+ 2R \sum\limits_{i = 1}^k \frac{1}{i^2} \\[10pt]
	& \leq &
	\|\xx_0 - \xx^{\star}\| + \sqrt{2A_0(F(\xx_0) - F^{\star})}
	+ 4R,
	\ea
	$$
	which completes the proof.
\end{proof}

Let us choose our parameters $R, A_0, \gamma$. We assume 
that $R$ and $A_0$ satisfy the following conditions:
\beq \label{AccParametersRA}
\ba{rcl}
R & \geq & \max\Bigl\{ \|\xx_0 - \xx^{\star}\|, \frac{2^{3/2}}{M} \Bigr\},
\qquad 
A_0 \;\; := \;\; \frac{c^2R^2}{2(F(\xx_0) - F^{\star})},
\ea
\eeq
for some absolute constant $c > 0$, and set
the contracting coefficient as:
\beq \label{AccGammaChoice}
\boxed{
\ba{rcl}
\gamma & := & \frac{1}{(MR)^{2/3}}
\ea
}
\eeq
Then, inequality \eqref{AccGuarantee} leads to
\beq \label{AccGuarantee2}
\ba{rcl}
A_k (F(\xx_k) - F^{\star}) + 
\frac{1}{2}\| \vv_k - \xx^{\star} \|^2
+ \frac{1}{2} \sum\limits_{i = 1}^k \| \vv_i - \vv_{i - 1} \|^2
& \leq & 
\frac{(5 + c)^2}{2}R^2.
\ea
\eeq
Therefore, we \textit{predefine} the \underline{accelerated global linear rate} for the iterations of our algorithm~\eqref{AcceleratedNewton}:
\beq \label{AccRate}
\ba{rcl}
F(\xx_k) - F^{\star} & \overset{\eqref{AccGuarantee2}}{\leq} &
\frac{(5 + c)^2 R^2}{2A_k}
\;\; \overset{\eqref{AkGrowth}}{\leq} \;\;
\exp\bigl( -\gamma k \bigr)\frac{(5 + c)^2 R^2}{2A_0} \\[10pt]
& \overset{\eqref{AccParametersRA},\eqref{AccGammaChoice}}{=} &
\exp\Bigl( - \frac{k}{(MR)^{2/3}}  \Bigr) \bigl( 1 + \frac{5}{c}  \bigr)^2
\bigl( F(\xx_0) - F^{\star} \bigr),
\qquad k \geq 1.
\ea
\eeq

Note that from~\eqref{AccGuarantee2}, we also conclude 
the boundedness for both our proximal sequence:
\beq \label{AccBoundedVk}
\ba{rcl}
\| \vv_k - \xx^{\star} \| & \overset{\eqref{AccGuarantee2}}{\leq} &
(5 + c) R, \qquad \forall k \geq 0,
\ea
\eeq
and, by induction, for the main sequence of points:
\beq \label{AccBoundedXk}
\ba{rcl}
\| \xx_{k + 1} - \xx^{\star} \| & \leq &
\gamma \|\vv_{k + 1} - \xx^{\star} \| + (1 - \gamma) \| \xx_k - \xx^{\star} \|
\;\; \overset{\eqref{AccBoundedVk}}{\leq} \;\; 
(5 + c)R, \qquad \forall k \geq 0.
\ea
\eeq

It remains to estimate the total number of iterations
of the Dual Newton Method, called at each step
of our accelerated scheme.
For that, we denote the Lipschitz constant
of the gradient $\nabla f(\cdot)$
and the Lipschitz constant
of the composite part $\psi(\cdot)$ 
over a compact convex set, respectively by
$$
\ba{rcl}
L_{f, 1} & \Def & 
\max\limits_{\xx \in Q}
\Bigl\{ \,
\| \nabla^2 f(\xx) \|
\; : \;
\| \xx - \xx^{\star} \| \leq (5 + c) R
\,
\Bigr\}
\ea
$$
and
$$
\ba{rcl}
L_{\psi, 0} & \Def & 
\max\limits_{\xx \in Q}
\Bigl\{ \,
\|  \psi'(\xx) \|_*
\; : \;
\psi'(\xx) \in \partial \psi(\xx), \;
\| \xx - \xx^{\star} \| \leq (5 + c) R
\,
\Bigr\}.
\ea
$$
Then, the subgradient $\| h'(\vv_k) \|_*$ that appears
under logarithm in \eqref{InnerMethodIters} can be globally bounded.
Let us denote $\yy_k = \gamma \vv_k + (1 - \gamma) \xx_k$.
We have (see, e.g., Theorem 2.1.5 in \cite{nesterov2018lectures} for the first inequality):
\beq \label{DiffGradBound}
\ba{rcl}
\frac{1}{2L_{f, 1}} \| \nabla f(\yy_k) - \nabla f(\xx^{\star}) \|_*^2
& \leq & 
f(\yy_k) - f(\xx^{\star}) - \la \nabla f(\xx^{\star}), \yy_k - \xx^{\star} \ra  \\[10pt]
& \leq &
F(\yy_k) - F^{\star}
\;\; \leq \;\; \mathcal{F}_0,
\ea
\eeq
where
$$
\ba{rcl}
\mathcal{F}_0 & \Def & 
\max\limits_{\xx \in Q} 
\Bigl\{ 
F(\xx) - F^{\star} \; : \; \|\xx - \xx^{\star} \| \leq (5 + c)R
\Bigr\}.
\ea
$$
Therefore, we obtain the bound
\beq \label{HSubgrBound}
\ba{rcl}
\| h'(\vv_k) \|_*
& = & 
a_{k + 1} \| \nabla f(\yy_k) + \psi'(\yy_k) \|_*
\;\; \leq \;\;
a_{k + 1} \bigl( 
\| \nabla f(\yy_k) - \nabla f(\xx^{\star}) \|_* + 2L_{\psi, 0}
\bigr) \\[10pt]
& \overset{\eqref{DiffGradBound}}{\leq} &
a_{k + 1}\bigl( 
\sqrt{2L_{f, 1} \mathcal{F}_0} + 2L_{\psi, 0} \bigr)
\;\; \overset{\eqref{AkGrowth}}{=} \;\;
\gamma A_0 
\bigl( 
\sqrt{2L_{f, 1} \mathcal{F}_0} + 2L_{\psi, 0} \bigr)
 (1 - \gamma)^{-(k + 1)} \\[10pt]
 & \leq &
 \gamma A_0 (\sqrt{2L_{f, 1} \mathcal{F}_0} + 2L_{\psi, 0} \bigr) e^{2\gamma(k + 1)}.
\ea
\eeq

Combining these observations with our \textit{hot-start}
estimate \eqref{InnerMethodIters} for the iterations of the inner method,
we prove the following result
about our accelerated scheme.

\begin{theorem} \label{TheoremAccelerated}
	Let the smooth component of problem~\eqref{MainProblem} be
	Quasi-Self-Concordant with parameter $M > 0$.
	Let the parameters $R$, $A_0$, $\gamma$
	be chosen according to~\eqref{AccParametersRA},\eqref{AccGammaChoice}.
	Then, the number of iterations of method~\eqref{AcceleratedNewton}
	to reach $\varepsilon$-accuracy~\eqref{RelAccuracyGuarantee} is bounded as
	\beq \label{AccKBound}
	\ba{rcl}
	k & \leq & \cO\Bigl( (MR)^{2/3} \cdot \ln \frac{1}{\varepsilon} \Bigr).
	\ea
	\eeq
	The total number of iterations of the inner Dual Newton Method
	is bounded as
	$$
	\ba{rcl}
	\sum\limits_{i = 1}^k I_i
	& \leq & \cO\Bigl( 
	(MR)^{2/3} \cdot \ln \frac{1}{\varepsilon} \cdot \mathcal{L}
	\Bigr),
	\ea
	$$
	where
	$$
	\ba{rcl}
	\mathcal{L} & \Def & 
	\ln\Bigl[ \, \frac{(MR)^{2/3}  R}{F(\xx_0) - F^{\star}} 
	\cdot ( \sqrt{L_{f, 1} \mathcal{F}_0} + L_{\psi, 0}  ) \cdot \frac{1}{\varepsilon} \ln^2 \frac{1}{\varepsilon}  \,\Bigr].
	\ea
	$$
\end{theorem}
\begin{proof}
	Bound \eqref{AccKBound}
	follows directly from~\eqref{AccRate}.
	Then, for the total number of iterations
	of the inner method, we have
	$$
	\ba{rcl}
	\sum\limits_{i = 1}^k I_i
	& \overset{\eqref{InnerMethodIters}}{\leq} &
	\cO\biggl( 
	\gamma^2 M^2 \sum\limits_{i = 1}^k \| \vv_{i} - \vv_{i - 1}\|^2
	+ \gamma^2 M^2 \sum\limits_{i = 1}^k \nu_i^2
	+ \sum\limits_{i = 1}^k \ln \frac{\| h'(\vv_i) \|_*}{\nu_i}
	\biggr) \\
	\\
	& \overset{\eqref{AccGuarantee2}}{\leq} &
	\mathcal{O}\biggl( 
	(\gamma MR)^2
	+ \gamma^2 M^2 \sum\limits_{i = 1}^k \nu_i^2
	+ \sum\limits_{i = 1}^k \ln \frac{\| h'(\vv_i) \|_*}{\nu_i}
	\biggr) \\
	\\
	& \leq & 
	\mathcal{O}\biggl( 
	(\gamma MR)^2
	+ \sum\limits_{i = 1}^k \ln \frac{i^2 \| h'(\vv_i) \|_* }{R}
	\biggr) \\
	\\
	& \overset{\eqref{HSubgrBound}}{\leq} &
	\mathcal{O}\biggl( 
	(\gamma MR)^2
	+ \gamma k^2
	+ \sum\limits_{i = 1}^k 
	\ln \frac{k^2 \gamma A_0 (\sqrt{L_{f, 1} \mathcal{F}_0} + L_{\psi, 0}) }{R}
	\biggr).
	\ea
	$$
	It remains to note that $(\gamma MR)^2 = \frac{1}{\gamma} = (MR)^{2/3}$
	and $\gamma k^2 = \frac{k^2}{(MR)^{2/3}} \overset{\eqref{AccKBound}}{\leq}
	\cO\Bigl(  (MR)^{2/3} \ln^2 \frac{1}{\varepsilon} \Bigr)$.
	Combining all estimates together completes the proof.
\end{proof}

\section{Discussion}
\label{SectionDiscussion}

In this work, we develop second-order optimization algorithms
for minimizing Quasi-Self-Concordant functions.
We prove that our methods possess fast global linear rates
of convergence on these problems, without additional
assumptions on strong or uniform convexity of the target objective.

The global linear rates for this problem class
were first proven in \cite{karimireddy2018global} using a trust-region scheme
based on the idea of \textit{Hessian stability}.
However, our algorithms use the Gradient Regularization technique
and involve only simple quadratic subproblems,
which are easier to implement and
can be more suitable for solving large-scale problems.

The primal Newton method~\eqref{PrimalNewton}
with implementation of the adaptive search matches
the recently proposed \textit{Super-Universal Newton Method}
(Algorithm~2 with $\alpha := 1$ in \cite{doikov2022super}).
It was shown in \cite{doikov2022super}
that this method automatically achieves 
good global convergence guarantees for wide classes
of functions with H\"older continuous second and third derivatives.
Therefore, our results also enhance complexity guarantees of the 
Super-Universal Newton Method
to Quasi-Self-Concordant functions.
Note that numerical experiments
in \cite{doikov2022super} demonstrated
that the regularization of the Newton Method with the first power of the gradient norm $(\alpha = 1)$
performs especially well in practice for solving Logistic Regression and Soft Maximum problems.

The dual Newton method~\eqref{DualNewton}
utilizes the local quadratic convergence of the primal Newton steps
and possesses a \textit{hot-start} convergence guarantee
for the dual objects (gradients).
Based on it, we develop the Accelerated Newton Scheme
which improves the global linear rate by taking the power $2/3$ of the condition number.

There are several open questions that are closely related to our results.
First, it would be interesting to compare our algorithms
with recently proposed
accelerated methods developed in the framework
of a \textit{ball-minimization oracle} \cite{carmon2020acceleration}.
The latter methods have the same power $2/3$ of the condition number as the main complexity factor,
but different logarithmic terms.
Moreover, it was also shown in \cite{carmon2020acceleration} that 
the power $2/3$ is \textit{optimal} for the methods based on the ball-minimization oracle.
Consequently, these optimality results could also hold significant implications for the minimization of Quasi-Self-Concordant functions.

Second, while the steps of our accelerated method are very easy to implement, 
its analysis requires fixing several parameters of the problem class,
which are: an estimate for the distance to the solution $R \geq \|\xx_0 - \xx^{\star}\|$, 
a relative distance with respect to the initial functional residual $A_0 \approx R^2 / ( {F(\xx_0) - F^{\star}} )$,
and the parameter of Quasi-Self-Concordance $M > 0$.
Although it is possible to use inexact estimates for these parameters,
the accelerated method itself \textit{predefines} the required liner rate by fixing the contracting coefficient  
$\gamma := 1 / (MR)^{2/3}$. This prevents the method
from adapting to the best problem class (as in the super-universal algorithms \cite{doikov2022super})
and from automatically switching to the local quadratic convergence.
Therefore, the development of more flexible versions of the accelerated algorithms would be highly beneficial in practice.

Finally, it seems to be very interesting to compare our results
with the recent advancements in the Path-Following schemes
for the classic Self-Concordant functions \cite{dvurechensky2018global,nesterov2023set}.
We keep these questions for further investigation.

\bibliographystyle{plain}
\bibliography{bibliography}

\end{document}